\documentclass{amsart}
\usepackage[utf8]{inputenc}
\usepackage{amsmath}
\usepackage{amsfonts}
\usepackage{hyperref}
\usepackage{graphicx}
\usepackage{tikz}
\usepackage{harpoon}
\usepackage{curves}
\usepackage{chngcntr}
\usepackage{apptools}
\usetikzlibrary{arrows}
\usepackage{enumerate}
\usepackage{float}
\usepackage{verbatim}
\usetikzlibrary{calc,decorations.markings}
\AtAppendix{\counterwithin{theorem}{section}}
\newtheorem{theorem}{Theorem}[section]

\newtheorem{corollary}[theorem]{Corollary}
\newtheorem{remark}[theorem]{Remark}

\newcommand{\R}{\mathbb{R}}
\newcommand{\Z}{\mathbb{Z}}
\newcommand{\N}{\mathbb{N}}
\newcommand{\C}{\mathbb{C}}

\newcommand{\D}{\mathbb{D}}
\newcommand{\Cn}{\mathbb{C}^n}
\newcommand{\Cnn}{\mathbb{C}^{n \times n}}
\newcommand{\Ct}{\mathbb{C}^2}
\newcommand{\Ctt}{\mathbb{C}^{2 \times 2}}
\newcommand{\Da}{\mathbb{D}_{\I a}}
\newcommand{\Dc}{\mathbb{D}_{\I c}}
\newcommand{\Dma}{\mathbb{D}_{-\I a}}
\newcommand{\Dmc}{\mathbb{D}_{-\I c}}
\newcommand{\md}{\text{mod}}
\newcommand{\nn}{\nonumber}
\newcommand{\beq}{\begin{equation}}
\newcommand{\eeq}{\end{equation}}
\newcommand{\bea}{\begin{eqnarray}}
\newcommand{\eea}{\end{eqnarray}}

\newcommand{\ol}{\overline}

\newcommand{\wti}{\widetilde}

\newcommand{\id}{\mathbb{I}}
\newcommand{\I}{\mathrm{i}}
\newcommand{\E}{\mathrm{e}}
\newcommand{\re}{\mathop{\mathrm{Re}}}
\newcommand{\im}{\mathop{\mathrm{Im}}}

\newcommand{\ga}{\gamma}
\newcommand{\om}{\omega}

\DeclareMathOperator{\Ai}{Ai}
\DeclareMathOperator{\res}{Res}
\DeclareMathOperator{\ai}{ai}
\DeclareMathOperator{\dd}{d}

\newcommand{\sigI}{\begin{pmatrix} 0 & 1 \\ 1 & 0 \end{pmatrix}}

\author[M. Piorkowski]{Mateusz Piorkowski}
\address{Department of Mathematics\\KU Leuven\\
Celestijnenlaan 200B 
\\
3001 Leuven, Belgium}
\email{\href{mailto:mathpiorkowski@gmail.com}{mathpiorkowski@gmail.com}}

\title[Parametrix problem for the Korteweg--de Vries equation]{Parametrix problem for the Korteweg--de Vries equation with steplike initial data}
\date{}
\thanks{Research supported by the Austrian Science Fund (FWF) under Grants No.\ P31651 and W1245.}
\begin{document}

\begin{abstract}
In this paper we study the asymptotics of solutions to the Korteweg--de Vries equation with steplike initial data, which lead to shock waves in the region between the asymptotically constant region and the soliton region, as $t \rightarrow \infty$. To achieve this, we present an alternative approach to the usual argument involving a small norm Riemann--Hilbert problem, which is based instead on the direct comparison of resolvents related to the corresponding Riemann--Hilbert problems. The motivation for this approach stems from the fact that an invertible holomorphic global parametrix solution for our problem does not exist for certain discrete times. 
\end{abstract}
\maketitle

\section{Introduction}
The Korteweg de--Vries (KdV) equation is one of the most investigated nonlinear wave equations that admits a Lax pair representation and thus can be solved via scattering theory. The explicit asymptotic analysis can be performed by the Deift--Zhou nonlinear steepest descent method for Riemann--Hilbert (R-H) problems (see \cite{DIZ}, \cite{DZ}, \cite{DZNLS}, \cite{DZPainleve}). It involves contour deformations and the introduction of auxiliary functions to obtain a R-H problem with jumps matrices that are either independent of the complex parameter $k$, or exponentially converging to the identity matrix for $t \rightarrow \infty$. Ignoring the exponentially converging part, one obtains a model problem, also referred to as the \emph{global parametrix problem}. In most applications it can be solved explicitly with the help of special functions (in our case Jacobi theta functions), from which the relevant asymptotics can be obtained. The rigorous justification of this method is however nontrivial and leads to a local R-H problem which has to be solved  around the oscillatory points (where the exponential convergence fails). The solution of this so-called \emph{local parametrix problem} again involves the use of special functions (in our case Airy functions) and needs to converge to the model solution locally uniformly away from the oscillatory points. 

The above steps for the KdV equation with steplike (shock wave) initial data in the region between the asymptotically constant and the soliton region, also called \emph{modulated elliptic wave region} after the form of the solutions, have been already performed in \cite{EGKT} and are summarized in the next section (see also \cite{EPT} for more details). The main topic of this paper is the final part of the analysis, which involves the construction of a invertible model matrix solution, or equivalently two linearly independent vector solutions. A peculiar feature of this step is, that while the relevant asymptotics can be read off from the symmetric model vector solution, the justification of the asymptotics requires the construction of a second linearly independent model vector solution (cf.~\cite{AELT}, \cite{PD},  \cite{GT}, \cite{NOEMA}). The reason for this is that inverting the model matrix solution results, after some analysis, in a singular integral equation of the form
\beq \label{SIE}
(\id - \mathcal{C}^\Sigma_u ) \phi = \mathcal{C}^\Sigma_u((1 \hspace{7pt} 1))
\eeq
where $\mathcal{C}^\Sigma_u$ is a singular Cauchy-type operator depending on $u$, where $u$  is a matrix-valued $L^\infty(\Sigma;\Ctt)$-function with $\Vert u \Vert_{L^\infty(\Sigma;\Ctt)} \rightarrow 0$, as $t \rightarrow \infty$ (see \cite[Ch.~7]{PD}). As $\Vert \mathcal{C}^\Sigma_u \Vert_{L^2(\Sigma;\Ct)} = O(\Vert u \Vert_{L^\infty(\Sigma;\Ctt)})$, we can invert $\id - \mathcal{C}^\Sigma_u$ for $t$ large enough by writing down the Neumann series. In particular, we know that equation \eqref{SIE} has a unique solution, which can be then used to write down the corresponding unique solution of the R-H problem. As the invertibility of the singular integral operator $\id - \mathcal{C}^\Sigma_u$ is obtained by the smallness of $\Vert u \Vert_{L^\infty(\Sigma;\Ctt)}$, we will refer to this approach as the \emph{small norm R-H approach}. 

The existence of a normalized second vector-valued solution fails in the case of interest for discrete but arbitrary large times, as we show in \cite{EPT}. One remedy would be the construction of a vector-valued meromorphic solution with a simple pole at the origin, as was done in \cite[Sect.~3]{GGJM} for an analogous model problem. Due to the inherit symmetries of the KdV R-H problem, the pole cancels in the final step of the nonlinear steepest descent analysis and leads to a familiar small norm R-H problem. 

In this work we avoid the small norm R-H problem and instead compare directly a modified version of the original R-H problem with a modified version of the model R-H problem. This is done by showing invertibility of one of the corresponding singular integral operators restricted to a suitable $L^2$-subspace, followed by a comparison of resolvents. This restriction is indeed necessary, as bijectivity does not hold on the full $L^2$-space for certain discrete, but arbitrary large, times $t$. This lack of bijectivity necessitated the construction of meromorphic solutions in \cite[Sect.~3]{GGJM}, and our approach also makes use of a meromorphic vector solution (with a pole at infinity).

The method for showing invertibility on the restricted space is based on the observation made in \cite[Sect.~2]{DZsob}, relating solvability of certain singular integral equations with the solvability of the corresponding inhomogeneous R-H problems. While we concentrate on the KdV case with steplike initial data, the approach taken in this paper should generalize to other problems solvable via the nonlinear steepest descent method.  

The structure of the paper is as follows. Section 2 summarizes the necessary scattering theory to obtain the R-H formulation of the KdV equation, as well as the conjugated and deformed R-H problem suitable for the nonlinear steepest descent analysis. Section 3 contains a local change of variables which results in an explicitly solvable local parametrix problem. The necessary theory of singular Cauchy-type integral operators with emphasis on the symmetries of our problem can be found in Section 4. The subsequent Section 5 contains the main idea of our new method: the construction of two auxiliary R-H problems which are suitable for a direct uniform comparison of the corresponding singular integral operators, leading to uniform error estimates for the KdV solution. It also contains the proof of invertibility of one of the singular integral operators on a suitable symmetric $L^2$-subspace, which is necessary for the resolvent comparison. The discussion section contains some further comments and a short scheme for obtaining the full asymptotic expansion of the KdV solution. The two appendices contain some proof technicalities left out in the main text and a general theorem which describes the method used in this paper. 
\section{preliminaries}

\subsection{Initial data}
We consider the KdV equation (cf.~\cite{EGKT}), given by
\begin{align*}
q_t(x,t) = 6 q(x,t)q_x(x,t) - q_{xxx}(x,t), \hspace{10pt} (x,t) \in \R \times \R_+
\end{align*}
with steplike initial data $q(x) = q(x,0) \in C^{11}(\R)$, i.e.
\begin{align*}
\lim\limits_{x \rightarrow \infty} q(x) &= 0, 
\\
\lim\limits_{x \rightarrow -\infty} q(x) &= -c^2, \hspace{10pt} c>0
\end{align*}
such that
\beq \label{fastDecay}
\int_0^{+\infty} \E^{C_0 x} (|q(x)| + |q(-x)+c^2|) dx < \infty, \hspace{10pt} C_0 > c
\eeq
and
\begin{align*}
\int_{-\infty}^\infty (x^6+1) |q^{(i)}(x)| dx < \infty, \hspace{10pt} i = 1, ..., 11.
\end{align*}
It has been shown (cf.~\cite{EGT2}, \cite{ET}) that the above Cauchy problem has a unique solution $q(\cdot,t) \in C^3(\R)$ satisfying additionally
\beq\label{fastDecay2}
\int_0^\infty |x|(|q(x,t)|+|q(-x,t)+c^2|)\, dx < \infty. 
\eeq
Existence of classical solutions has been proven under more general assumptions in \cite{Ryb}, but we require the more restrictive condition \eqref{fastDecay} for analytic continuation of the reflection coefficients in the framework of the nonlinear steepest descent method, and the estimate \eqref{fastDecay2} for the construction of Jost solutions and well-posedness of the corresponding R-H problem. We will focus on the asymptotic behaviour of solutions in the modulated elliptic wave region given by $-6c^2t < x < 4c^2t$. For simplicity, we will assume that the solution $q(x,t)$ develops no solitons. While the technicalities coming from considering solitons are minor, they would not influence the asymptotics in the region of interest as the corresponding model R-H problem remains unchanged (see \cite[Sect.~5]{EGKT}). 
\subsection{Scattering transform}
The necessary scattering theory for steplike potentials can be found in \cite[Sect.~2]{EGKT} (see also \cite[Sect.~2]{EPT} for more details) and shall be reviewed here briefly. To solve the KdV equation via the scattering transform, we need to regard the solution $q(x,t)$ as a potential of a self-adjoint Schr\"odinger operator:
\begin{align*}
L(t) = -\dfrac{d^2}{dx^2} + q(\, \cdot \,,t), \hspace{10pt} \mathfrak{D}(L) = H^2(\R) \subset L^2(\R).
\end{align*}
Because of the behaviour of $q(x,t)$ for $x \to \pm \infty$, one can find unique Jost solutions $\phi(k,x,t)$, $\phi_1(k,x,t)$ of the stationary Schr\"odinger equation
\begin{align*}
L(t) \psi(k,x,t) = k^2 \psi(k,x,t), \hspace{10pt} \im(k) >0,
\end{align*}
determined by
\begin{align*}
\lim\limits_{x \rightarrow \infty} \E^{-\I k x} \phi(k,x,t) = 1, \hspace{10pt} \lim\limits_{x \rightarrow -\infty} \E^{\I k_1 x} \phi_1(k,x,t) = 1,
\end{align*}
where $k_1 := \sqrt{k^2 + c^2}$ is holomorphic in $\C \setminus [-\I c, \I c]$ with $k_1 > 0$ for $k > 0$. We endow $[-\I c, \I c]$ with an orientation from top to bottom, hence $+$ $(-)$ denotes the limit from the right (left), e.g. $k_{1,+} = - k_{1.-}$. The Jost solutions $\phi(k,x,t)$ and $\phi_1(k,x,t)$ are holomorphic in the variable $k$ in the domain $\C^U := \lbrace k: \im(k) > 0\rbrace $, and $\C^U_c :=\C^U \setminus (0, \I c]$ respectively and continuous up to the boundary. Hence, we can evaluate $\phi(k,x,t)$ and $\phi_1(k,x,t)$ on the real axis, which results in the scattering relations
\begin{align*}
T(k,t)\phi_1(k,x,t) &= \overline{\phi(k,x,t)} + R(k,t) \phi(k,x,t), \hspace{36pt} k \in \R,
\\
T_1(k,t) \phi(k,x,t) &= \overline{\phi_1(k,x,t)}+R_1(k,t) \phi_1(k,x,t), \qquad k_1 \in \R,
\end{align*}
where $T(k,t)$, $T_1(k,t)$ and $R(k,t)$, $R_1(k,t)$ are the transmission and reflection coefficients determined uniquely by the above equations. In the case of the absence of the discrete spectrum, $T(k,t)$ and $T_1(k,t)$ are holomorphic in $\C^U_c$ and continuous up to the boundary, while $R(k,t)$ and $R_1(k,t)$ have an analytic continuation to the domain $\lbrace k : 0 <\im (k) < C_0 \rbrace \setminus (0,\I c] $, because of assumption \eqref{fastDecay}.
Moreover, we will assume the generic nonresonant case (for an analysis including the resonant case, see \cite[Sect.~2]{EPT}), meaning that the Wronskian $W(k,t) := W(\phi, \phi_1)(k,t)$ does not vanish at $k = \I c$. As the transmission coefficients can be related to the Wronskian via
\beq\nonumber
T(k,t) = \frac{2\I k}{W(k,t)}, \qquad T_1(k,t) = \frac{2\I k_1}{W(k,t)}, 
\eeq
we see that $T(k,t)^{\pm 1}$ is bounded in the vicinity of $\I c$, while $T_1(k,t) = O((k-\I c)^{1/2})$ as $k \rightarrow \I c$. We also introduce an auxiliary function
\begin{align} \label{chi}
\chi(k,t) := -\lim\limits_{\varepsilon \rightarrow 0+} \overline{T(k+\varepsilon,t)} T_1(k+\varepsilon,t), \hspace{10pt} k \in (0,\I c]
\end{align}
and extend it to $[-\I c, 0)$ via
\begin{align*}
\chi(-k,t) = - \chi(k,t).
\end{align*}
More properties of the above functions can be found in \cite[Sect.~2]{EGKT} (see also \cite[Sect.~2]{EPT}), and shall be mentioned when needed. 

The next step involves defining a minimal scattering data, from which the potential $q(\cdot,t)$ can be reconstructed. The choice of the scattering data is determined by the form of the Gelfand--Levitan--Marchenko equation (cf.~\cite{CK85}) and has the form
\begin{align*}
S(t) = \lbrace R(k,t), \ k \in \R; \ \chi(k,t), \ k\in (0, \I c]\rbrace.
\end{align*}
Here, $S(t)$ denotes the scattering data of the solution $q(\cdot,t)$ of the KdV equation, which evolves linearly from the scattering data $S(0)$ of the initial data $q(\cdot,t)$ via 
\begin{align}\label{timeevol}
\begin{split}
R(k,t) &= R(k,0) \E^{8\I k^3 t} = R(k) \E^{8\I k^3 t},
\\
\chi(k,t) &= \chi(k,0) \E^{8 \I k^3 t} = \chi(k) \E^{8 \I k^3 t}.
\end{split}
\end{align}
We see that the direct scattering transform, mapping $q(\cdot ,t) \to S(t)$, effectively linearizes the KdV equation. The R-H approach is then used to perform the inverse scattering transform, mapping $S(t) \rightarrow q(\cdot,t)$, and is outlined in the following theorem taken from \cite[Thm.~2.3]{EGKT} (see also \cite[Thm.~2.1]{EPT} for the proof of uniqueness):
\begin{theorem} \label{RHPfor_m}
Let $m(k) = m(k,x,t)$ be given by
\beq\nonumber
m(k,x,t)=
\begin{cases}
\Big(T(k,t) \phi_1(k,x,t) \E^{\I kx},  \phi(k,x,t) \E^{-\I kx}\Big),  & k\in \C^U_c
\\
\Big(\phi(-k,x,t) \E^{\I kx}, T(-k,t) \phi_1(-k,x,t) \E^{-\I kx}\Big),  & k\in\C^L_c,
\end{cases}
\eeq
where $\C^U_c=\{k:\ \im k>0\}\setminus(0,\I c]$,
$\C^L_c=\{k:\ \im k<0\}\setminus(0,-\I c]$. Then $m(k)$ is the unique solution to the following R-H problem:
\\
\\
Find a vector-valued function $m(k) = (m_1(k), \, m_2(k))$, which is holomorphic in $\C \setminus (\R \cup [-\I c,\I c]) $, and has continuous limits to $\R \cup [-\I c,\I c]$, satisfying:
\begin{enumerate}[(i)]
\item The jump condition $m_+(k)=m_-(k) v(k)$,
\beq \nonumber
v(k)=\left\{\begin{array}{cc}\begin{pmatrix}
1-|R(k)|^2 & - \ol{R(k)} \E^{-t\Phi(k)} \\
R(k) \E^{t\Phi(k)} & 1
\end{pmatrix},& k\in\R,\\
 \ &\ \\
\begin{pmatrix}
1 & 0 \\
\chi(k) \E^{t\Phi(k)} & 1
\end{pmatrix},& k\in (0, \I c],\\
 \ &\ \\
\begin{pmatrix}
1 & \chi(k) \E^{-t\Phi(k)} \\
0 & 1
\end{pmatrix},& k\in [-\I c, 0),\\
\end{array}\right.
\eeq
where $\R$ is oriented from left to right and $[-\I c, \I c]$ is oriented from top to bottom,
\item
the symmetry condition
\beq \label{eq:symcond}
m(-k) = m(k) \sigI,
\eeq
\item
and the normalization condition
\beq\label{eq:normcond}
\lim_{k\to\infty} m(k) = (1 \hspace{7pt} 1).
\eeq
\end{enumerate}
Here the phase $\Phi(k)=\Phi(k,x,t)$ is given by
\begin{equation*}
\Phi(k)= 8 \I k^3+2\I k \frac {x}{t},
\end{equation*}
\end{theorem}
Note that the jump matrix $v(k)$ also satisfies a symmetry condition
\beq \label{vSym}
v(-k) = \sigma_1 v(k)^{-1} \sigma_1, \quad 
\sigma_1 := 
\begin{pmatrix}
0 & 1
\\
1 & 0
\end{pmatrix}.
\eeq
There are two methods to obtain $q(x,t)$ from $m(k,x,t) = (m_1(k,x,t), \, m_2(k,x,t))$ (see \cite{EPT}):
\begin{align*}
q(x,t) =  \partial_x \lim\limits_{k \rightarrow \infty} 2 \I k(m_1(k,x,t)-1) = -\partial_x \lim\limits_{k \rightarrow \infty} 2 \I k ( m_2(k,x,t)-1),
\end{align*}
\beq \label{mproduct}
q(x,t) = \lim\limits_{k \rightarrow \infty} 2 k^2(m_1(k,x,t) m_2(k,x,t)-1).
\eeq
The first formulas are more analytically demanding because of the differentiation, so we will use the second formula \eqref{mproduct}.

\subsection{Conjugation steps}

For further analysis we introduce the following function
\beq\nonumber
{g(k)=g(k,x,t):=12\int_{\I c}^k (k^2 + \mu^2)\sqrt\frac{k^2 + a^2}{k^2 + c^2}} dk
\eeq
which is holomorphic in $\C \setminus [-\I c, \I c]$ and approximates $\Phi(k)$ at infinity, while simplifying our R-H problem on $[-\I c, \I c]$. As has been shown in \cite[Sect.~4]{KM10}, $a = a(\xi)$ and $\mu = \mu(\xi)$ can be chosen to depend continuously on the slowly varying parameter $\xi = \frac{x}{12t} \in (-\frac{c^2}{2}, \frac{c^2}{3})$ such that the following  properties hold (see \cite[Sect.~4]{EGKT}):
\begin{enumerate}[(i).]
\item $\mbox{The function } g \mbox{ is odd, i.e. } g(-k) = -g(k), \hspace{3pt} k \in \C \setminus [-\I c, \I c]$;
\item $g_-(k) + g_+(k) = 0$ for  $k \in [-\I c, \I c] \setminus (-\I a, \I a)$;
\item $g_-(k) - g_+(k) = B$ for $k \in [-\I a, \I a]$, with $B := -2g_+(\I a) > 0$;
\item for $k \rightarrow \infty$ we have
\begin{align*}
\dfrac{1}{2}\Phi(k, \xi) - \I g(k, \xi) =  O(k^{-1}).
\end{align*}
\end{enumerate}

In the following we summarize the necessary factorization and conjugation steps (see \cite[Sect.~4]{EGKT} for details). We start by modify the R-H problem from Theorem \ref{RHPfor_m} by conjugating by the matrix $\E^{-(t \Phi(k)/2-\I t g(k)) \sigma_3}$, i.e.
\begin{align*}
m(k) \longrightarrow \hat m(k) &:= m(k) \E^{-(t \Phi(k)/2-\I t g(k)) \sigma_3},
\\
v(k) \longrightarrow \hat v(k) &:= \E^{(t \Phi(k)/2-\I t g_-(k)) \sigma_3} v(k) \E^{-(t \Phi(k)/2-\I t g_+(k)) \sigma_3},
\\
\sigma_3 &:=
\begin{pmatrix}
1 & 0
\\
0 & -1
\end{pmatrix} .
\end{align*}
The next step involves performing a standard factorization of the jump matrix on the real axis and further local conjugations. The resulting R-H problem has the jump matrix
\beq\label{v22}
v^{(2)}(k) = \left\{ \begin{array}{ll}
\begin{pmatrix}
0 & \I \\
\I  & 0
\end{pmatrix},& k\in [\I c, \I a],
\\
\\
\begin{pmatrix}\E^{-\I t \hat B}& 0\\
\frac{\I|\chi|}{F_-^2}\E^{\I t(2g_-- \hat{B})} &\E^{\I t\hat{B}}\end{pmatrix}, &k\in[\I a , \I b]
\\
\\
\begin{pmatrix}
1&0\\
\tfrac{R}{F^2}\E^{2\I t g} & 1
\end{pmatrix}, & k\in\Sigma^U
\\
\\
\begin{pmatrix}\E^{-\I t \hat B}& 0\\
0&\E^{\I t \hat B}\end{pmatrix},
& k\in [-\I a, \I a],\\
\\
\\
\begin{pmatrix}
1 & \tfrac{F^2}{V}\E^{-2\I t g}
\\
0 & 1
\end{pmatrix}, & k\in\Sigma^U_1,
\\
\\
\sigma_1 \big[v^{(2)}(-k)]^{-1}\sigma_1, & k \in [-\I a, -\I c] \cup [-\I b, -\I a]\cup \Sigma^L\cup \Sigma^L_1,
\end{array}\right.
\eeq
with the corresponding jump contour given in Figure \ref{fig3}. Note that we use the superscript ${}^{(2)}$ to make the comparison with \cite{EGKT} easier, though our $v^{(2)}(k)$ differs slightly from the one found therein due to a different choice of the jump contour. We require the solution $m^{(2)}(k,x,t) = m^{(2)}(k)$ to satisfy the usual symmetry and normalization conditions \eqref{eq:symcond}, \eqref{eq:normcond}, and to have bounded limits to the contour, except for $k \to \pm \I c$, where we allow (see \cite[Eq.~(4.30)]{EGKT})
\beq\nn
m^{(2)}(k) = O((k\mp \I c)^{-1/4}), \qquad k \to \pm \I c.
\eeq
\begin{figure}[h]
\begin{picture}(8,4.5)
\put(4,0){\line(0,1){4}}
\put(4,3.6){\vector(0,-1){0.1}}
\put(4,1.7){\vector(0,-1){0.1}}
\put(4,0.5){\vector(0,-1){0.1}}

\put(3.6,3.05){$\I a$}
\put(3.35, 0.8){$-\I a$}
\put(4.1, 4){$\I c$}
\put(4, -0.2){$-\I c$}
\put(4.15,2.75){$\I b$}
\put(4.1,1.05){$-\I b$}
\put(4,0.8){\circle*{0.1}}
\put(4,3.2){\circle*{0.1}}

\put(4,0){\circle*{0.1}}
\put(4,4){\circle*{0.1}}
\put(4, 2.8){\circle*{0.1}}
\put(4, 1.2){\circle*{0.1}}

\put(2,2.5){$\Sigma^U$}
\put(2,1.2){$\Sigma^L$}
\put(5.5,2.5){$\Sigma^U$}
\put(5.5,1.2){$\Sigma^L$}

\put(5.1,4){$\Sigma^U_1$}
\put(5.1,-0.2){$\Sigma^L_1$}

\put(3.4,2.25){$\Omega^U$}
\put(3.4,1.5){$\Omega^L$}

\put(3.4,3.6){$\Omega_1^U$}
\put(3.4,0.2){$\Omega_1^L$}

\curve(4,2.8, 5,2.5, 7.5,2.3)
\curve(4,1.2, 5,1.5, 7.5,1.7)
\curve(4,2.8, 3,2.5, 0.5,2.3)
\curve(4,1.2, 3,1.5, 0.5,1.7)

\put(6.5,2.34){\vector(1,0){0.1}}
\put(6.5,1.66){\vector(1,0){0.1}}
\put(1.5,2.34){\vector(1,0){0.1}}
\put(1.5,1.66){\vector(1,0){0.1}}

\curve(4,3.2, 3,3.9, 4,4.5, 5,3.9, 4,3.2)
\curve(4,0.8, 3,0.1, 4,-0.5, 5,0.1, 4,0.8)

\put(4.8,3.62){\vector(1,1){0.1}}
\put(3.1,3.7){\vector(1, -1){0.1}}
\put(4.85,0.35){\vector(1,-1){0.1}}
\put(3.1,0.3){\vector(1,1){0.1}}

\curvedashes{0.05,0.05}
\curve(0.3,2, 7.7,2)

\end{picture}
\vspace{10pt}
  \caption{The contour $\Sigma$ of the model R-H problem with exponential correction}\label{fig3}
\end{figure}
In \eqref{v22}, we have that $\hat{B} := B + \frac{\Delta}{t}$, 
\beq\nonumber
\Delta = \Delta(\xi) = 2\I \frac{\int_{\I a}^{\I c} \frac{\log |\chi(s)|\, ds}{[\sqrt{(s^2+c^2)(s^2+a^2)}]_+}}{\int_{-\I a}^{\I a}\frac{ds}{\sqrt{(s^2+c^2)(s^2+a^2)}}} \in \R
\eeq
and $F(k)$ is an auxiliary function defined in \cite[Eq.~(4.21)]{EGKT} having the properties:
\begin{enumerate}[(i).]
\item $F_+(k)F_-(k)=|\chi(k)|$ for $k\in [\I c,\I a]$,
\item $F_+(k)F_-(k)=|\chi(k)|^{-1}$ for $k\in[-\I a,-\I c]$,
\item $F_+(k)=F_-(k)\E^{\I \Delta} $ for $k\in[-\I a, \I a]$,
\item $F(k)\to 1$ as $k\to \infty$ and $F(-k)=F^{-1}(k)$ for $k\in\C\setminus [-\I c, \I c]$,
\item $F(k)$ is holomorphic in $\C \setminus [-\I c, \I c]$ with $F(k) = O((k\mp\I c)^{\pm 1/4})$ for $k \to \pm \I c$ (see \cite[p.~17]{EPT}).
\end{enumerate}
Moreover, $V(k) := \overline{T(k)}T_1(k)$ for $k \in \C^U_c$ and $V(-k) = V(k)$ for $k \in \C^L_c$, note $V_+(k) = -\chi(k)$ for $k \in (0,\I c]$.

We write $\Sigma$ for the union of all contours listed in \eqref{v22}, and we assume that $\Sigma$ is invariant under the map $k \to -k$, which reverses its orientation as depicted in Figure \ref{fig3}. Note that the jump matrix $v^{(2)}(k)$ satisfies the symmetry condition \eqref{vSym}.  
As all conjugation and deformation steps are invertible, we know from Theorem \ref{RHPfor_m}, that there exists a unique solution $m^{(2)}(k)$ to the R-H problem with jump matrix $v^{(2)}(k)$, together with the usual asymptotics at infinity \eqref{eq:normcond} and the symmetry condition \eqref{eq:symcond}. Moreover, we assume that $\Sigma^U$ and $\Sigma^L$ remain a finite distance away from the real line, implying that for $t \rightarrow \infty$ the jump matrices on $\Sigma^U$, $\Sigma^L$, $\Sigma_1^U \setminus \lbrace \I a \rbrace$ and $\Sigma_1^L \setminus \lbrace -\I a \rbrace$ converge pointwise exponentially fast to the identity matrix, and on $(\I a, \I b] \cup (-\I a, -\I b]$ to the diagonal matrix $\E^{-\I t \hat{B} \sigma_3}$. Note however, that the exponential convergence rate is not uniform for $k \to \pm \I a$. We shall refer to the R-H problem for $m^{(2)}(k)$ as the \emph{model R-H problem with exponential correction}, and the one where we ignore the  jump matrices converging to the identity matrix as the \emph{asymptotic model R-H problem} or just \emph{model R-H problem}. The latter one is solved explicitly in \cite{EGKT} (see also \cite{KM10}) and its solution is unique (see \cite[Sect.~3]{EPT}):
\begin{theorem} \label{modUniq}
The model R-H problem, given through the jump matrix 
\begin{align}\label{modJump}
v^{\emph{mod}}(k) = 
\left\{ \begin{array}{ll}
\begin{pmatrix}
0 & \I \\
\I  & 0
\end{pmatrix},& k\in [\I a, \I c],\\
\\
\begin{pmatrix}
0 & -\I \\
-\I & 0
\end{pmatrix},& k\in [-\I a, - \I c],\\
\\
\begin{pmatrix}\E^{-\I t \hat B}& 0\\
0&\E^{\I t \hat B}\end{pmatrix},
& k\in [-\I a, \I a],\\
\end{array}\right.
\end{align}
with the jump contour $[-\I c, \I c]$ oriented from top to bottom, has a unique solution $m^{\emph{mod}}(k)$ with continuous limits to $(-\I c, \I c) \setminus \lbrace -\I a, \I a \rbrace$, satisfying the symmetry and normalization conditions \eqref{eq:symcond}, \eqref{eq:normcond}, and having at most fourth root singularities around the points $\kappa \in \lbrace -\I a, \I a, -\I c, \I c \rbrace$:
\beq\nonumber
m^{\emph{mod}}(k) = O((k-\kappa)^{-1/4}), \quad k \to \kappa.
\eeq
\end{theorem}

\section{Parametrix problem}
\subsection{Local change of variables}

We now turn to the jump condition near the points $\pm \I a$ (we will just consider $+\I a$, analogous results hold for $-\I a$). For $k$ near $\I a$ we can write

\begin{align*}
g(k) &= 12 \int_{\I c}^k (s^2 + \mu^2) \sqrt{\dfrac{s^2 + a^2}{s^2 + c^2}} ds\\
&= g_\pm(\I a) + 12 \int_{\I a}^{k} (s^2 + \mu^2) \sqrt{\dfrac{s^2 + a^2}{s^2 + c^2}} ds\\\nn
& = \mp \dfrac{B(\xi)}{2} - 8 \E^{\I\pi/4} (a^2 - \mu^2) \sqrt{\dfrac{2a}{c^2-a^2}} (k-\I a)^{3/2} + O((k- \I a)^{5/2}),
\end{align*}
where the roots in the last two lines have a branch cut on the positive imaginary axis (i.e.~$\im k > a$, $\re k = 0$). We choose all roots such that they are positive for positive arguments $k-\I a > 0$. The upper (lower) sign is for the limit
from the right (left), respectively.

Next, we define a local time-dependent holomorphic change of variables $k-\I a \rightarrow w$, such that
\beq \nonumber
g(k) = \mp \dfrac{B(\xi)}{2} + \dfrac{\varsigma w(k)^{3/2}}{t}, \qquad k \in \Da
\eeq
where $\Da$ is a disc around $\I a$ with a fixed radius smaller than $\min(c-a, a-b)$, such that $k-\I a \to w$ is bijective and $\varsigma = \E^{-3\pi \I/4}$. The branch cut is defined again on the positive imaginary axis.  Furthermore, if convenient we will abuse notation by writing $f(w)$ for $f(k(w))$, if $f$ is a function of the variable $k$ and vice versa. 

We continue with the model R-H problem with exponential correction around the critical point $\I a$. To make it independent of our new variable $w$, we perform a conjugation by the matrix $\E^{-\I t g(w) \sigma_3}$. This results in local  jump conditions with the following jump matrices around $k = \I a$ ($w = 0$), where $\hat h := \E^{\I \Delta}$:
\begin{figure}[H] 
\begin{picture}(7,5.7)
\put(3.5,0){\line(0,1){5}}
\put(3.5,1){\vector(0,-1){0.1}}
\put(3.5, 4){\vector(0,-1){0.1}}
\put(1,5){\line(1,-1){2.5}}
\put(2,4){\vector(1,-1){0.1}}
\put(3.5,2.5){\line(1,1){2.5}}
\put(5,4){\vector(1,1){0.1}}

\put(5,3.5){$\scriptsize\begin{pmatrix} 1 & F^2/V \\ 0 & 1 \end{pmatrix}$}
\put(0.1,3.5){$\scriptsize\begin{pmatrix} 1 & F^2/V \\ 0 & 1 \end{pmatrix}$}

\put(3.6,4.5){$\scriptsize\begin{pmatrix} 0 & \I 
\\
\I   & 0 \end{pmatrix}$}
\put(3.6,0.7){$\scriptsize\begin{pmatrix} \hat{h}^{-1} & 0 
\\ 
\hat{h} \chi/F_+^2 & \hat{h} \end{pmatrix}$}
\end{picture}
\caption{The local jump conditions}
\end{figure}

In the next step we conjugate around the origin in the $w$-domain by the matrix
\begin{align*}
\begin{pmatrix}
p(w) & 0
\\
0 & p(w)^{-1}
\end{pmatrix}
\end{align*}
where $p(w)$ should be a holomorphic function with a possible jump on the imaginary axis. The jump matrices transform as follows
\vspace{10pt}
\begin{align*}
\begin{pmatrix} 1 & F^2/V \\ 0 & 1 \end{pmatrix}
&\longrightarrow
\begin{pmatrix} 1 &  F^2/V p^{-2} \\ 0 & 1 \end{pmatrix}
\\
\begin{pmatrix} 0 & \I 
\\
\I   & 0 \end{pmatrix}
&\longrightarrow
\begin{pmatrix} 0 & \I p_+^{-1} p_-^{-1}
\\
\I p_+ p_-  & 0 \end{pmatrix}
\\
\begin{pmatrix} \hat{h}^{-1} & 0 
\\ 
\hat{h} \chi/F_+^2 & \hat{h} \end{pmatrix}
&\longrightarrow
\begin{pmatrix} \hat{h}^{-1} p_+ p_-^{-1} & 0 
\\ 
\hat{h} \chi/F_+^2  p_+ p_- & \hat{h} p_+^{-1}p_- \end{pmatrix}.
\end{align*}
To simplify the problem, we require that 
\beq \nn
\dfrac{F^2}{V}p^{-2} = 1, \qquad \mbox{in } \ \Da \setminus (0, \I c]
\eeq
which implies 
\beq \nn
p = \dfrac{F}{\sqrt{V}}, \qquad \mbox{in } \ \Da \setminus (0,\I c].
\eeq
As the limit of $F^2/V$ from the right (left) to $\I a$ is equal to $\I \hat{h}$ ($-\I \hat{h}^{-1}$) which are of modulus $1$, we see that we can find locally a square root. We choose the normalization
\beq \nn
p_+(0) = \E^{\pi \I /4} \hat{h}^{1/2} 
\eeq
and
\beq \nn
p_-(0) = \E^{-\pi \I /4} \hat{h}^{-1/2}.
\eeq
The following boundary values can be computed explicitly:
\beq \nn
p_+ p_- = \dfrac{F_+ F_-}{\sqrt{V_+}\sqrt{V_-}} =  \dfrac{|\chi |}{\sqrt{-\chi}\sqrt{\chi}} = \dfrac{\I |\chi|}{\chi}  =   1, \hspace{10pt} k(w) \in [\I a, \I c] \cap \Da,
\eeq
\beq \nn
p_+ p_-^{-1} = \dfrac{F_+ F_-^{-1}}{\sqrt{V_+}\sqrt{V_-}^{-1}} = \dfrac{\hat{h}}{\sqrt{-\chi}\sqrt{\chi}^{-1}} = \I \hat{h}, \hspace{10pt} k(w) \in [0, \I a] \cap \Da.
\eeq
The conjugated jump conditions take on the form:
\begin{figure}[H] 
\begin{picture}(7,5.7)

\put(3.5,0.5){\line(0,1){4.5}}
\put(3.5,1.5){\vector(0,-1){0.1}}
\put(3.5, 4){\vector(0,-1){0.1}}
\put(1.2,4.8){\line(1,-1){2.3}}
\put(2,4){\vector(1,-1){0.1}}
\put(3.5,2.5){\line(1,1){2.3}}
\put(5,4){\vector(1,1){0.1}}

\put(2.2, 2.2){$\Omega_4$}
\put(4.6, 2.2){$\Omega_1$}
\put(4, 3.7){$\Omega_2$}
\put(2.7, 3.7){$\Omega_3$}

\put(5,3.5){$\scriptsize\begin{pmatrix} 1 & 1 \\ 0 & 1 \end{pmatrix}$}
\put(0.9,3.5){$\scriptsize\begin{pmatrix} 1 & 1 \\ 0 & 1 \end{pmatrix}$}

\put(3.6,4.5){$\scriptsize\begin{pmatrix} 0 & \I 
\\
\I    & 0 \end{pmatrix}$}
\put(3.6,1.3){$\scriptsize\begin{pmatrix} \I & 0 
\\ 
\I  & -\I \end{pmatrix}$}

\put(5.8,5){$\Sigma_1$}
\put(0.8,5){$\Sigma_3$}

\put(3.3, 5.3){$\Sigma_2$}
\put(3.3, 0.1){$\Sigma_4$}
\end{picture}

\caption{The conjugated jump conditions}
\end{figure}
where we have used that
\beq\nonumber
\dfrac{\hat{h} \chi p_+ p_-}{F_+^2} = \I 
\eeq
on the imaginary segment $k(w) \in [0, \I a] \cap \Da$. The problem has, up to conjugation, the same jump structure as the \emph{Airy R-H problem} (see e.g.~, \cite[Sect.~6]{AELT}, \cite{PD}, \cite[Sect.~5.3]{GGJM}), except that in our setting we do not require any normalization at infinity. 

An explicit solution can be given in terms of Airy functions and their derivatives (see \cite[Ch.~9]{OLVER}). As this solution is well-known, we will skip its construction,  referring instead to \cite[Sect.~3.2]{HarnardBook}. What will be important to us is the form of the local Airy parametrix solution $A(k)$ in the $k$-domain around the point $\I a$. For convenience we also conjugate back with $p(w)^{-\sigma_3}$ and $\E^{\I t g(w)\sigma_3}$, where $w = w(k)$, resulting in:
\beq \label{airyMatrix}
A(k) = 
\begin{cases}
\E^{\I \pi/4} t^{-\sigma_3/6}
\begin{pmatrix}
\frac{2\I}{3\varsigma} & 0
\\
0 & 1
\end{pmatrix}
\begin{pmatrix}
y_3'(w) & -y_1'(w)
\\
 -y_3(w) & y_1(w)
\end{pmatrix}p(w)^{-\sigma_3}\E^{\I t g(w) \sigma_3},  \hspace{14pt} w \in \Omega_1,
\\
\\
\E^{ \I \pi/4} t^{-\sigma_3/6}
\begin{pmatrix}
\frac{2\I}{3\varsigma} & 0
\\
0 & 1
\end{pmatrix}
\begin{pmatrix}
y_3'(w) & y_2'(w)
\\
-y_3(w) & -y_2(w) 
\end{pmatrix}p(w)^{-\sigma_3}\E^{\I t g(w) \sigma_3}, \hspace{14pt} w \in \Omega_2,
\\
\\
\E^{\I \pi/4} t^{-\sigma_3/6}
\begin{pmatrix}
\frac{2\I}{3\varsigma} & 0
\\
0 & 1
\end{pmatrix}
\begin{pmatrix}
-\I y_2'(w) & -\I y_3'(w)
\\
\I y_2(w) & \I y_3(w) 
\end{pmatrix}p(w)^{-\sigma_3}\E^{\I t g(w) \sigma_3}, \hspace{10pt} w \in \Omega_3,
\\
\\
\E^{\I \pi/4} t^{-\sigma_3/6}
\begin{pmatrix}
\frac{2\I}{3\varsigma} & 0
\\
0 & 1
\end{pmatrix}
\begin{pmatrix}
-\I y_2'(w) & \I y_1'(w)
\\
\I y_2(w) & -\I y_1(w) 
\end{pmatrix}p(w)^{-\sigma_3}\E^{\I t g(w) \sigma_3}, \hspace{10pt} w \in \Omega_4,
\end{cases}
\eeq
in the disc $\Da$ specified earlier. Here the functions $y_i(w)$, $i = 1, \dots 3$ are defined to be
\begin{align*} \nonumber
y_1(w) &:=  2 \sqrt{\pi} \E^{\I \pi /8} (3/2)^{1/6}\Ai(\I w (3/2)^{2/3}),
\\
y_2(w) &:= \rho y_1(\rho w),
\\
\nonumber
y_3(w) &:= \rho^2 y_1(\rho^2 w),
\end{align*}
where $\rho = \E^{2\pi \I/3}$.
Note that $m^{(2)}(k)A^{-1}(k)$ will have no jumps inside $\Da$. 

For our subsequent analysis we need an analogous local solution for the model problem, i.e.~we look for a matrix $N(k)$ defined on $\Da$, such that $m^\md(k)N^{-1}(k)$ has no jumps inside $\Da$. This translate to the following jump condition for $N(k)$ 
\begin{equation} \label{jumpN}
\begin{aligned}
N_+(k) &=  N_-(k)\begin{pmatrix}
0 & \I
\\
\I & 0
\end{pmatrix}, \hspace{40pt} k \in [\I a, \I c] \cap \Da,
\\
N_+(k) &=  N_-(k)\begin{pmatrix}
\E^{-\I t \hat{B}} & 0
\\
0 & \E^{\I t \hat{B}}
\end{pmatrix}, \hspace{10pt} k \in [0, \I a] \cap \Da.
\end{aligned}
\end{equation}
Furthermore, on the boundary $\partial\Da$ we require $N(k)$ to asymptotically converge to $A(k)$. 
\begin{remark}
To the best of the authors knowledge, such a R-H problem for a local solution $N(k)$ of the global model problem is new. In many cases, a global model matrix-valued solution is available (see e.g. \cite[Sect.~5]{AELT}, \cite[Sect.~4.4]{DKMVZ}, \cite[Sect.~4]{KM10}, \cite[Sect.~5]{KMVV}), making the construction of a local version $N(k)$ obsolete.
\end{remark}
The correct solution of \eqref{jumpN} is given by
\beq \label{N}
N(k) =  \frac{t^{-\sigma_3/6}}{\sqrt{2}}\begin{pmatrix}
 w^{1/4}  & w^{1/4}
\\
-w^{-1/4}  & w^{-1/4} 
\end{pmatrix}\E^{\mp\I(t B/2 + \pi/4)\sigma_3} p(w)^{-\sigma_3}, \quad k \in \Da \setminus (0, \I c]
\eeq
which is obtained from $A(k)$ by taking the first term in the expansion of the Airy function and their derivative (cf.~\cite[Eq.~9.7.5--6]{OLVER}). 
Note that in $A (k)$ the $\exp(\I \varsigma w(k)^{3/2})$ factor cancels partially with $\exp(\pm \I t g(k)) = \exp(\mp \I B/2 + \I \varsigma w(k)^{3/2})$ leaving only $\exp(\pm \I t B/2)$, which is contained in formula \eqref{N}. In fact, choosing $N(k)$ as above and using the normalization in \eqref{airyMatrix} we obtain the estimate
\beq\label{AandN}
A(k) = N(k) +
O(t^{-1}), \quad k \in \partial \Da,
\eeq
which follows from the asymptotic expansion and the fact that $O(t^{1/6}w^{-7/4}) = O(t^{-1/6} w^{-5/4}) = O(t^{-1})$ on $\partial \Da$. As both the determinants of $A(k)$ and $N(k)$ are constant equal to $1$, we also have the estimate
\beq\label{AN-1}
A^{-1}(k) = N^{-1}(k) + O(t^{-1}), \quad k \in \partial \Da.
\eeq  

In the Section 5 we will also need a local solution $Q(k)$ of the model problem around the point $\I c$.  We choose a fixed disc $\Dc$ around the point $\I c$ with radius small enough, so that $\Da \cap \Dc = \Sigma^U_1 \cap \Dc = \emptyset$, and set
\beq\label{Q}
   Q(k) = 
   \frac{1}{\sqrt{2}}\begin{pmatrix}
   (k-\I c)^{1/4} & -(k-\I c)^{1/4}
   \\
   (k-\I c)^{-1/4} & (k-\I c)^{-1/4}
   \end{pmatrix}, \quad k \in \Dc
\eeq
where the fourth roots have a branch cut on $\Dc \cap [\I a, \I c]$, and the choice of the branch is irrelevant. As we will see in Section 5, the exact form of $Q(k)$ for $k \in \partial \Dc$ will not matter. Note that the determinant of $Q(k)$ is constant equal to $1$.

\section{Singular integral equations}
Our goal is to show that the contributions coming from the vicinities of $\pm \I a$ are small, such that they do not affect the leading asymptotics of the KdV equation. This can be achieved by reformulating our R-H problems as a singular integral equations and is a rigorous justification of Theorem 5.1 in \cite{EGKT}, giving uniform error estimates of order $O(t^{-1})$. Again, the arguments follow a similar line to the ones given in \cite{AELT}, \cite{EPT} and \cite{NOEMA}, except for the final analysis, which omits the construction of a global matrix-valued model solution and the small norm R-H problem. Instead, we show directly the invertibility of the corresponding singular integral operators, when restricted to certain symmetric $L^2$-spaces. Relevant literature on Cauchy operators and their connection with R-H problems can be found in \cite{BK97}, \cite{PD}, \cite{JL0}. We shall review the essential results here. 

We write $\mathcal{C}^{\Gamma}$ for the Cauchy operator defined on $L^2(\Gamma)$,
\beq\nn
\mathcal{C}^\Gamma: L^2(\Gamma) \rightarrow \mathcal{O}(\C \setminus \Gamma), \hspace{10pt} f(k) \mapsto \mathcal{C}^\Gamma(f)(k) := \dfrac{1}{2\pi \I} \int_\Gamma f(s) \dfrac{ds}{s-k},
\eeq
where $\mathcal O(\C \setminus \Gamma)$ denotes the space of holomorphic functions on $\C \setminus \Gamma$. Here, we are assuming that $\Gamma$ has an orientation and that the family of functions $(s-k)^{-1}$ are in $L^2(\Gamma)$ for $k \in \C \setminus \Gamma$. We define the operators $\mathcal{C}^\Gamma_-$, $\mathcal{C}^\Gamma_+$ by
\beq\nn
\mathcal{C}^\Gamma_\pm(f)(k) = \lim\limits_{z \rightarrow k \pm} \mathcal{C}^\Gamma(f)(z),
\eeq
where the limit is assumed to be nontangential. Standard theory tells us that if $\Gamma$ is a Carleson contour, the nontangential limit exists a.e.~and $\mathcal{C}^\Gamma_\pm$ will be a bounded operator from $L^2(\Gamma)$ to itself (see \cite{BK97}, \cite[Prop.~3.11]{JL0}). 

For a general $2 \times 2$ matrix-valued function $u(k)\in L^\infty(\Gamma;\Ctt)$ we can define with slight abuse of notation the following operator:
\beq\nn
\mathcal{C}^{\Gamma}_u : L^2(\Gamma;\Ct)\rightarrow L^2(\Gamma;\Ct), \hspace{10pt} f \mapsto \mathcal{C}^\Gamma_-(f \cdot u),
\eeq
where $\Gamma$ is a Carleson jump contour. Here, $\mathcal{C}^\Gamma_-$ acts componentwise on vector-valued entries. We assume that $\Gamma$ is invariant with respect to the transformation $k \rightarrow -k$, and that sequences converging to the positive side still converge after this transformation to the positive side. This is a different convention then the one we used in Section 2 and Theorem \ref{RHPfor_m}, in order to simplify computations that follow. The jump matrix $v(k)$ (and also $u(k) := v(k) - \id$) must now satisfy
\beq\nn
v(-k) = 
\sigma_1
v(k)
\sigma_1, \qquad k \in \Gamma,
\eeq
which is the same as \eqref{vSym} when taking into account the different orientation. By the previous arguments, we have that $\Vert\mathcal{C}^{\Gamma}_u \Vert_{L^2(\Gamma;\Ct)} \leq C \Vert u \Vert_{L^\infty(\Gamma;\Ctt)}$, where $C$ is the operator bound of $\mathcal{C}^{\Gamma}_-$ in $L^2(\Gamma;\Ct)$.  
\begin{remark}
We emphasise that the orientation is essentially a free choice. While some theorems in the literature (c.f.~\cite{JL0}) require particular orientation conventions, one can always fulfill them by changing orientations of certain arcs and inverting the corresponding jump matrix. 
\end{remark}

Next, let us consider the following integral equation
\beq \nonumber
(\id - \mathcal{C}^\Gamma_u)\phi(k) = \mathcal{C}^\Gamma_-((1 \hspace{7pt} 1)u), 
\eeq
where we also require that the matrix entries of $u(k)$ are in $L^2(\Gamma)$ in order for the right-hand side to be well-defined. It turns out that there is a bijective correspondence between solutions of the above equation, and vector solutions of the R-H problem on the contour $\Gamma$ with jump matrices $v(k) = \id + u(k)$ (in the $L^2$-setting) given by
\begin{equation}
\nonumber
\begin{aligned}
\phi(k) &\longrightarrow m(k) := (1 \hspace{7pt} 1) + \dfrac{1}{2 \pi \I} \int_\Gamma (\phi(s) + (1 \hspace{7pt} 1)) u(s) \dfrac{ds}{s-k}
\\
m(k) &\longrightarrow \phi(k) := m_-(k) - (1 \hspace{7pt} 1),
\end{aligned}
\end{equation}
see \cite{XZ}. 

As the uniqueness statements in Theorems \ref{RHPfor_m} and \ref{modUniq} only hold for vector solutions satisfying the symmetry condition \eqref{eq:symcond}, we would need to restrict the operator $\mathcal{C}^{\Gamma}_u$ to $L^2_s(\Gamma;\Ct)$, which we define to be the Hilbert space of functions $\phi(k)\in L^2(\Gamma;\Ct)$, satisfying
\beq\nn
\phi(-k) = \phi(k)
\sigma_1.
\eeq
Hence, we define
\beq \nonumber
\mathcal{S}^\Gamma_u \colon L^2_s(\Gamma;\Ct) \rightarrow L^2_s(\Gamma;\Ct), \quad f \mapsto \mathcal{C}^\Gamma_u f
\eeq
to be the restriction of $\mathcal{C}^\Gamma_u$ to the space of symmetric functions $L^2_s(\Gamma;\Ct)$. Analogously, $\mathcal{A}^\Gamma_u$ denotes the restriction of $\mathcal{C}^\Gamma_u$ to the space $L^2_a(\Gamma;\Ct)$ of  antisymmetric functions satisfying
\beq\nn
\phi(-k) = -\phi(k)
\sigma_1.
\eeq
To see that $\mathcal{S}_u^\Gamma$, $\mathcal{A}_u^\Gamma$ are well-defined, i.e.~that $\mathcal{C}_u^\Gamma$ indeed maps $L^2_s(\Gamma;\Ct)$ ($L^2_a(\Gamma;\Ct)$) to $L^2_s(\Gamma;\Ct)$ ($L^2_a(\Gamma;\Ct)$), we  define the operator
\begin{align*}\nn
H \colon L^2(\Gamma;\Ct) \rightarrow L^2(\Gamma;\Ct),
\qquad
H\phi(k) := \phi(-k)
\sigma_1.
\end{align*}
Note that $L^2_s(\Gamma;\Ct)$ is the eigenspace of $H$ with eigenvalue $1$, while $L^2_a(\Gamma;\Ct)$ is the eigenspace with eigenvalue $-1$. Next, let us assume that $u(k)$ (or analogously $v(k)$) satisfies the symmetry condition
\beq\nn
u(-k) = 
\sigma_1
u(k)
\sigma_1.
\eeq
We can then compute that $H$ is a symmetry of $\mathcal{C}^\Gamma_u$, i.e.~commutes with it
\begin{align*}
\begin{split}
\mathcal{C}^\Gamma_u H \phi (k) &= \lim\limits_{k' \rightarrow k-}\dfrac{1}{2\pi \I} \int_\Gamma H\phi(s) u(s) \dfrac{ds}{s-k'}
\\
&= \lim\limits_{k' \rightarrow k-}\dfrac{1}{2\pi \I} \int_\Gamma \phi(-s) 
\sigma_1
u(s) \dfrac{ds}{s-k'}
\\
&= \lim\limits_{-k' \rightarrow -k-}\dfrac{1}{2\pi \I} \int_\Gamma \phi(s) 
u(s) \dfrac{ds}{s-(-k')}
\sigma_1
\\
&=
\mathcal{C}^\Gamma_u \phi(-k)
\sigma_1
\\
&=
H\mathcal{C}^\Gamma_u \phi(k).
\end{split}
\end{align*}
Hence, we conclude
\beq\nn
[\mathcal{C}^\Gamma_u, H] = 0 
\eeq
which implies that the range of $\mathcal{S}^\Gamma_u$ lies in the space of symmetric functions, while the range of $\mathcal{A}^\Gamma_u$ lies in the space of antisymmetric functions, as claimed. 
\section{Main result}

Next, we define two new R-H problems for which we can write down their unique solutions. Recall that $\Da$ a disc around $\I a$ with radius smaller than $\min(c- a, a-b)$, such that $k \rightarrow w$ is bijective for $k \in \Da$. In particular, $\Sigma^U$ is assumed to be  disjoint from $\Da$. Note that the radius can be chosen to be constant with respect to small variations of $\xi$ (see Remark \ref{rmkxi}). Let $\Dma$ be the image of $\Da$ under the transformation $k \rightarrow -k$ and assume that $\partial \Da$ and $\partial \Dma$ are oriented counterclockwise. Analogous properties are assumed for $\Dc$ and $\Dmc$, which are small discs centered at $\I c$ and $-\I c$ respectively, where the radius is chosen small enough to avoid  intersections with $\Sigma^U_1$ and $\Sigma^L_1$. We define $\mathbb{D}_\cup := \Da \cup \Dma \cup \Dc \cup \Dmc$. As anticipated in the last section, we invert the orientation of $\Sigma \cap \lbrace z \colon \im z < 0 \rbrace$ to simplify the analysis. 

One can check that the two R-H problems given below will satisfy the symmetry condition for the contour and for the jump matrices specified in Section 4. The same goes for the solutions, which are assumed to satisfy condition \eqref{eq:symcond}. As we want to use the integral operator theory from the previous section, all limits to the contour are assumed to hold locally in the $L^2$-sense. This limit notion turns out to be equivalent to the more classical notion of continuous limits with at most fourth root singularities, as used in Theorem \ref{modUniq}. The justification for this equivalence will be given in the proof of Theorem 5.6, and is for now assumed.

\subsection{Riemann--Hilbert problem I}
Find a vector-valued function $m^{I}(k)$, holomorphic in $\C \setminus \Sigma^{I}$, where $\Sigma^{I} = ([-\I c, \I c]\setminus  \mathbb{D}_\cup)  \cup \partial \mathbb{D}_\cup$, satisfying: 
\begin{enumerate}[(i).]
\item  The jump condition $m^{I}_+(k) = m^{I}_-(k)v^{I}(k)$:
\beq\nn
v^{I}(k) = 
\begin{cases}
v^\md(k), \quad & k \in (0, \I c] \setminus \mathbb{D}_\cup
\\
[v^\md(k)]^{-1}, \quad & k \in [-\I c, 0) \setminus \mathbb{D}_\cup
\\
N^{-1}(k), \quad & k\in \partial \Da
\\
\sigma_1 N^{-1}(-k) \sigma_1, \quad & k \in \partial \Dma
\\
Q^{-1}(k), \quad & k\in\partial \Dc
\\
\sigma_1 Q^{-1}(-k) \sigma_1, \quad & k\in\partial \Dmc
\end{cases}
\eeq
with $N(k)$ and $Q(k)$ defined by \eqref{N} and \eqref{Q},
\item the symmetry condition
\beq\nn
m^{I}(-k) = m^{I}(k) 
\begin{pmatrix}
0 & 1
\\
1 & 0
\end{pmatrix},
\eeq
\item 
and the normalization condition
\beq\nn
\lim\limits_{k \rightarrow \infty} m^{I}(k)= (1 \hspace{7pt} 1).
\eeq
\end{enumerate}
The unique solution has the form
\begin{align} \label{solutionRHP1}
m^{I}(k) =
\begin{cases}
 m^{\text{mod}}(k), & k \in \C \setminus \mathcal \D_\cup
\\
 m^{\text{mod}}(k)N^{-1}(k), & k \in \Da
\\
m^{\text{mod}}(k)\sigma_1 N^{-1}(-k)\sigma_1, & k \in \Dma
\\
 m^{\text{mod}}(k)Q^{-1}(k), & k \in \Dc
\\
m^{\text{mod}}(k)\sigma_1 Q^{-1}(-k)\sigma_1, & k \in \Dmc
\end{cases}
\end{align}
where $m^\md(k)$ satisfies the model R-H problem from Theorem \ref{modUniq} and is taken from \cite{EGKT}. Indeed, $m^I(k)$ defined in \eqref{solutionRHP1} has at most square root singularities at the points $\pm \I c$, $\pm \I a$, but no jumps inside $\D_{\cup}$, implying that it is in fact analytic inside $\D_{\cup}$. 

To write down the solution $m^\md(k)$, we need to introduce the genus one Riemann surface $\mathbb{X}$, which is obtained by gluing two copies (sheets) of $\C \setminus ([-\I c, -\I a] \cup [\I a, \I c])$ along the cuts $[-\I c, -\I a]$ and $[\I a, \I c]$. The corresponding canonical homology basis of cycles $\lbrace \bf a, \bf b \rbrace$ is defined as follows: The $\bf a$-cycle starts on the upper sheet from the left side of
the cut $[\I a,\I c]$, continues on the upper sheet to the left part of $[-\I c, -\I a]$ and returns after changing sheets. The $\bf b$-cycle surrounds the cut $[\I a, \I c] $ counterclockwise on the upper sheet.

Having defined the homology basis, the normalized holomorphic differential $d\omega$ is defined by
\beq\nonumber
d\om=2\pi\I\left(\int_{\bf a} \frac{d \zeta}{\sqrt{(\zeta^2+c^2)(\zeta^2+a^2)}}\right)^{-1}\frac{d\zeta}{\sqrt{(\zeta^2+c^2)(\zeta^2+a^2)}},
\eeq
where the square root has branch cuts on $[-\I c, -\I a]\cup [\I a, \I c]$, and is positive at $\zeta = 0$. Note that $\int_{\bf a}d\om=2\pi\I$. It turns out that the Riemann surface $\mathbb{X}$ can be characterized by the half-period ratio $\tau$ given by
\begin{align*}
    \tau=\int_{\bf b} d\om< 0
\end{align*}
and is biholomorphically equivalent to $\C / \Lambda$, where $\Lambda = \lbrace 2\pi\I n  + \tau \ell : n, \ell \in \mathbb{Z} \rbrace$. To explicitly write down the biholomorphism, we need to define the Abel map:
\begin{align*}
    \mathbf{A}: \mathbb{X} \rightarrow \C / \Lambda, \quad p \mapsto \int_{\I c}^p d\omega.
\end{align*}
However, in our R-H setting, we will restrict the path of integration to $\C \setminus [-\I c, -\I c]$ on the upper sheet, in which case $\mathbf A = \mathbf A(k)$ defines a mapping from $\C \setminus [-\I c, -\I c]$ to $\C$ and satisfies the following properties (cf.~\cite[Sect.~5]{EGKT}):
\begin{enumerate}[(i).]
    \item $\mathbf A_l(k) = -\mathbf A_r(k) \ \mod(2\pi \I), \ \ k \in [-\I c, -\I a]\cup [\I a, \I c]$
    \item $\mathbf A_l(k) - \mathbf A_r(k) = \tau, \ \ k \in [-\I a, \I a]$,
    \item $\mathbf A(-k) = -\mathbf A(k)+\pi \I, \ \ k\in \C \setminus [-\I c, \I c]$,
    \item $\mathbf A_l(k) =  \frac{\tau}{2} + O(\sqrt{k-\I a}), \ \ \mathbf A_l(k) =  \frac{\tau}{2} +\pi \I + O(\sqrt{k+\I a}) \ $ as $k\to \pm \I a$,
    \item $\mathbf A(\infty) = \frac{\pi\I}{2}$.
\end{enumerate}
Here the subscript $l$ (resp.~$r$) denotes the left (resp.~right) limit to $[-\I c, \I c]$.

The final ingredient for writing down the explicit solution to the model R-H problem is the Jacobi theta function:
\begin{align}\label{thetaFct}
    \theta(z) = \theta(z|\tau) = \sum_{n \in \mathbb{Z}} \exp \Big\lbrace \frac{1}{2} \tau n^2 + n z  \Big\rbrace,
\end{align}
which satisfies the quasiperiodicity condition:
\begin{align}\label{quasiper}
    \theta(z+2\pi\I n+ \tau \ell) = \theta(z)\exp \Big\lbrace -\frac{1}{2} \tau \ell^2 - \ell z  \Big\rbrace
\end{align}
and has simple zeros precisely at $z \equiv \pi\I + \tau/2  \mod \Lambda$. More generally, corresponding to any $\tau$ with $\re \tau < 0$, there are four Jacobi theta functions, with the one in \eqref{thetaFct} being commonly denoted by $\theta_3(z)$ in the literature. For a thorough treatment of Jacobi theta functions and their various properties and notational conventions, see \cite[Sect.~6]{M67}.

The model solution to the R-H problem from Theorem \ref{modUniq} can be found in \cite[Sect.~5]{EGKT} and is given by
\begin{align}\label{modsol}
\begin{split}
 m_1^{\text{\text{mod}}}(k,t) &= \sqrt[4]{\frac{k^2 + a^2}{k^2+c^2}}\frac{\theta\left(\mathbf A(k) - \I\pi-\frac{\I t\hat B}{2}\right)\theta\left(\mathbf A(k) -\frac{\I t\hat B}{2}\right)\theta^2\left(\frac{\pi\I}{2}\right)}
{{\theta\left(\mathbf A(k) - \I\pi\right)\theta\left(\mathbf A(k) \right)\theta\left(\frac{\pi\I}{2}-\frac{\I t\hat B}{2}\right)\theta\left(\frac{\pi\I}{2}+\frac{\I t\hat B}{2}\right)}},\\
m_2^{\text{\text{mod}}}(k,t) &= \sqrt[4]{\frac{k^2 + a^2}{k^2+c^2}}\frac{\theta\left(-\mathbf A(k) - \I\pi-\frac{\I t\hat B}{2}\right)\theta\left(-\mathbf A(k) -\frac{\I t\hat B}{2}\right)\theta^2\left(\frac{\pi\I}{2}\right)}
{{\theta\left(-\mathbf A(k) - \I\pi\right)\theta\left(-\mathbf A(k) \right)\theta\left(\frac{\pi\I}{2}-\frac{\I t\hat B}{2}\right)\theta\left(\frac{\pi\I}{2}+\frac{\I t\hat B}{2}\right)}},
\end{split}
\end{align}
where the fourth root has branch cuts on $[-\I c, -\I a]\cup[\I a, \I c]$ and converges to $1$ at infinity. Properties (i) and (ii) of the Abel map $\mathbf{A}(k)$ together with the quasiperiodicity condition \eqref{quasiper} of $\theta(z)$, imply the jump condition \eqref{modJump}. Property (iii) implies the symmetry condition \eqref{eq:symcond}, while property (iv) together with the location of simple zeros of $\theta(z)$ guarantees that $m^\md(k)$ defined in \eqref{modsol} has at most fourth root singularities at the contour. Property (v) implies the normalization \eqref{eq:normcond}. For a step by step derivation of  \eqref{modsol} see \cite{PT}. Note that uniqueness of $m^{\text{mod}}(k)$ implies uniqueness of $m^{I}(k)$, as any model vector solution would give rise to a solution to R-H problem I via \eqref{solutionRHP1}.

\subsection{Riemann--Hilbert problem II}
Find a vector-valued function $m^{II}(k)$, holomorphic in $\C \setminus \Sigma^{II}$, where $\Sigma^{II} = (\Sigma \setminus \mathbb{D}_\cup) \cup \partial \mathbb{D}_\cup$, satisfying: 
\begin{enumerate}[(i).]
\item  The jump condition $m^{II}_+(k) = m^{II}_-(k)v^{II}(k)$:
\beq\nn
v^{II}(k) = 
\begin{cases}
v^{(2)}(k), \quad & k\in (\Sigma \setminus \D_\cup)\cap \lbrace z \colon \im z > 0 \rbrace
\\
[v^{(2)}(k)]^{-1}, \quad & k\in (\Sigma \setminus \D_\cup)\cap \lbrace z \colon \im z < 0 \rbrace
\\
A^{-1}(k), \quad & k\in \partial \Da
\\
\sigma_1 A^{-1}(-k) \sigma_1, \quad & k \in \partial \Dma
\\
Q^{-1}(k), \quad & k\in\partial \Dc
\\
\sigma_1 Q^{-1}(-k) \sigma_1, \quad & k\in\partial \Dmc
\end{cases}
\eeq
with $A(k)$ and $Q(k)$ defined by \eqref{airyMatrix} and \eqref{Q},
\item the symmetry condition
\beq\nn
m^{II}(-k) = m^{II}(k) 
\begin{pmatrix}
0 & 1
\\
1 & 0
\end{pmatrix},
\eeq
\item 
and the normalization condition
\beq\nn
\lim\limits_{k \rightarrow \infty} m^{II}(k)= (1 \hspace{7pt} 1).
\eeq
\end{enumerate}
The unique solution has the form
\begin{align} \label{solutionRHP2}
m^{II}(k) =
\begin{cases}
 m^{(2)}(k), & k \in \C \setminus \mathbb{D}_\cup
\\
 m^{(2)}(k)A^{-1}(k), & k \in \Da
\\
m^{(2)}(k)\sigma_1 A^{-1}(-k)\sigma_1, & k \in \Dma
\\
 m^{(2)}(k)Q^{-1}(k), & k \in \Dc
\\
m^{(2)}(k)\sigma_1 Q^{-1}(-k)\sigma_1, & k \in \Dmc
\end{cases}
\end{align}
where again uniqueness of $m^{(2)}(k)$ implies uniqueness of $m^{II}(k)$, which is analytic in $\D_\cup$. 

Let us remark once again that $N^{-1}(k)$, $Q^{-1}(k)$ and $A^{-1}(k)$ are chosen to cancel the jump matrices in $\mathbb{D}_\cup$. While R-H problem I has only jumps on the imaginary segments $[-\I c, \I c] \setminus \mathbb{D}_\cup$ and $\partial \mathbb{D}_\cup$, R-H problem II additionally has jumps uniformly converging to the identity matrix away from the discs.

Note that $v^I(k) = v^{II}(k)$ for $k \in [-\I c, \I c] \setminus \big(\D_\cup \cup [\I a, \I b]\cup[-\I a, -\I b]\big)$ and for $k \in \partial\Dc \cup \partial\Dmc$. On $\partial\Da \cup \partial\Dma$ we have the estimate \eqref{AN-1}, hence we can conclude that
\beq\nn
\Vert v^{II} - v^{I} \Vert_{L^\infty(\Sigma^{II};\Ctt)} = O(t^{-1}),
\eeq
where we set $v^{I}(k) \equiv \id$ for $k \in \Sigma^{II} \setminus \Sigma^{I}$. 

\subsection{Invertibility of $\id - \mathcal{S}^{\Sigma^{II}}_{u^{I}}$}
We will now use the bijection between solutions of R-H problems and solutions to singular integral equations, stated in Section 4. In particular, we will analyse the singular integral operator $\id - \mathcal{S}^{\Sigma^{II}}_{u^{I}}$ which appears in the equation
\beq\nn
(\id - \mathcal{S}^{\Sigma^{II}}_{u^{I}})\phi^I(k) = \mathcal{C}^{\Sigma^{II}}_-((1 \hspace{7pt} 1)u^{I}), 
\eeq
where $\phi^I(k) = m^{I}_-(k)-(1 \hspace{7pt} 1)$. Recall that we set $u^I(k) = v^I(k)-\id=0$ for $k\in\Sigma^{II} \setminus \Sigma^I$. The following result taken from \cite[Sect.~2]{DZsob} will play a crucial role:
\begin{theorem}\label{ThmInh}
Assume $v(k) \in L^\infty(\Gamma;\Cnn)$ is an $n \times n$ matrix-valued function and $\Gamma$ is a Carleson contour. Then the singular integral operator
\beq
\id - \mathcal C_u^\Gamma \colon L^2(\Gamma;\Cn) \rightarrow L^2(\Gamma;\Cn)
\eeq
with $u(k) = v(k)-\id$ is bijective, if and only if for each $\lambda(k) \in L^2(\Gamma;\Cn)$ there is a unique solution to the following inhomogeneous R-H problem: 

Find a vector-valued function $\mu(k)$, holomorphic for $k\in\C \setminus \Gamma$, satisfying
\begin{align*}
    \mu_+(k) &= \mu_-(k) v(k) + \lambda(k), \quad \mu_\pm(k) \in L^2(\Gamma;\Cn),
    \\
    \lim_{k\to\infty} \mu(k) &= 0.
\end{align*}
\end{theorem}
Moreover, we will also use the following corollary, again taken from \cite[Sect.~2]{DZsob}:
\begin{corollary}\label{LongCor}
Assume $v(k) \in L^\infty(\Gamma_1;\Cnn)$ is an $n \times n$ matrix valued function and $\Gamma_\cup = \Gamma_1  \dot\cup \, \Gamma_2$, where $\Gamma_i$ for $i=1,2$ are Carleson contours. Define 
\begin{align}
   \widetilde v(k) =
   \begin{cases}
    v(k) & k\in \Gamma_1
    \\
    \id & k\in \Gamma_2
   \end{cases}, \quad \widetilde u(k) = \widetilde v(k)-\id.
\end{align}
Then the operator $\id-\mathcal C^{\Gamma_1}_u$ is invertible if and only the operator $\id - \mathcal C^{\Gamma_\cup}_{\widetilde u}$ is invertible.
\end{corollary}
\begin{proof}
We will use the characterization from Theorem \ref{ThmInh}. First assume $\id-\mathcal C^{\Gamma_1}_u$ is invertible and take $\widetilde \lambda(k) \in L^2(\Gamma_\cup;\Cn)$ arbitrary. We know that the inhomogeneous R-H problem
\begin{align*}
    \mu_+(k) &= \mu_-(k)v(k) + \mathcal{C}^{\Gamma_\cup}_-(\wti \lambda)(k)u(k), \quad \mu_\pm(k) \in L^2(\Gamma_1;\Cn),
    \\
    \lim_{k\to \infty} \mu(k) &= 0, 
\end{align*}
has a unique solution $\mu(k)$ holomorphic for $k \in \C \setminus \Gamma_1$. Define $\widetilde \mu(k) = \mu(k) + \mathcal{C}^{\Gamma_\cup}(\widetilde \lambda)(k)$. Then one can compute that $\wti \mu(k)$ satisfies
\begin{align*}
    \wti \mu_+(k) &= \wti \mu_-(k)\wti v(k) + \wti \lambda(k), \quad \wti \mu_\pm(k) \in L^2(\Gamma_\cup;\Cn),
    \\
    \lim_{k\to \infty} \widetilde \mu(k) &= 0. 
\end{align*}

Now assume that $\id - \mathcal C^{\Gamma_\cup}_{\wti u}$ is invertible. For any fixed $\lambda(k) \in L^2(\Gamma;\Cn)$ define
\begin{align*}
    \wti \lambda(k) = \begin{cases}
     \lambda(k), & k \in \Gamma_1
     \\
     0, & k \in \Gamma_2
    \end{cases}.
\end{align*}
It follows that the unique $\wti \mu(k)$ corresponding to the inhomogeneity $\wti \lambda(k)$ will solve
\begin{align*}
    \wti \mu_+(k) &= \wti \mu_-(k) v(k) + \lambda(k), \quad k \in \Gamma_1
    \\
    \wti \mu_+(k) &= \wti \mu_-(k), \quad k \in \Gamma_2
    \\
    \lim_{k\to0} \wti \mu(k) &= 0, \qquad \wti \mu_\pm(k)|_{k \in \Gamma_1} \in L^2(\Gamma_1).
\end{align*}
Hence we can choose $\mu(k) = \wti \mu(k)$ for $k \in \C \setminus \Gamma_\cup$, and $\mu(k) = \wti \mu_\pm(k)$ for $k \in \Gamma_2$, to be a solution of the inhomogeneous R-H problem with inhomogeneity $\lambda(k)$. 

So far we have shown that solvability of the two inhomogeneous R-H problems is equivalent. The statement about unique solvability now follows from the simple observation that the corresponding vanishing problems (taking $\lambda(k) \equiv 0$ on $\Gamma_1$, $\wti \lambda(k) \equiv 0$ on $\Gamma_\cup$) reduce to exactly the same homogenous R-H problem with jump contour $\Gamma_1$. 
\end{proof}
\begin{remark}
The statement of Theorem \ref{ThmInh} and Corollary \ref{LongCor} remains true if we assume the symmetries described in Sections 4 to hold, and consider the operator $\id - \mathcal S_u^\Gamma$ ($\id - \mathcal A_u^\Gamma$) instead, provided we require the inhomogeneity $\lambda(k)$ to be in $L^2_s(\Gamma;\Ct)$ ($L^2_a(\Gamma;\Ct)$) as well. 
\end{remark}
Corollary \ref{LongCor} implies that showing the invertibility of $\id - \mathcal{S}^{\Sigma^{II}}_{u^{I}}$ is equivalent to showing the invertibility of $\id - \mathcal{S}^{\Sigma^{I}}_{u^{I}}$, which is the operator we will consider next. This can be achieved by constructing a second linearly independent vector-valued solution of the model problem. Equivalently we are looking for an invertible matrix-valued solution to the model problem. We will consider an antisymmetric vector solution $n^\md(k)$ satisfying:
\begin{align} \label{anticond}
    n^{\text{mod}}(-k) = -n^{\text{mod}}(k)\sigma_1.
\end{align}
An explicit formula for such an $n^\md(k)$ is given by
\begin{align*}
    n_1^{\text{mod}}(k,t) &= \sqrt[4]{\frac{k^2 + a^2}{k^2+c^2}}\frac{\theta\left(\mathbf A(k) - \I\pi-\frac{\I t\hat B}{2}-\frac{\tau}{2}\right)\theta\left(\mathbf A(k) -\frac{\I t\hat B}{2}+\frac{\tau}{2}\right)}{{\theta\left(\mathbf A(k) - \I\pi\right)\theta\left(\mathbf A(k) \right)}},\\
n_2^{\text{mod}}(k,t) &= \sqrt[4]{\frac{k^2 + a^2}{k^2+c^2}}\frac{\theta\left(-\mathbf A(k) - \I\pi-\frac{\I t\hat B}{2}-\frac{\tau}{2}\right)\theta\left(-\mathbf A(k) -\frac{\I t\hat B}{2}+\frac{\tau}{2}\right)}
{{\theta\left(-\mathbf A(k) - \I\pi\right)\theta\left(-\mathbf A(k) \right)}}.
\end{align*}
Note that unlike $m^\md(k,t)$, the vector-valued function $n^{\text{mod}}(k,t)$ is not normalized at $k = \infty$. The reason is that 
\begin{align*}
     \lim_{k\to\infty} n_1^{\text{mod}}(k,t) = - \lim_{k\to\infty} n_2^{\text{mod}}(k,t) =  \frac{\theta\left(\mathbf - \frac{\I\pi}{2}-\frac{\I t\hat B}{2}-\frac{\tau}{2}\right)\theta\left(\frac{\I \pi}{2} -\frac{\I t\hat B}{2}+\frac{\tau}{2}\right)}{\theta^2\left(\frac{\I\pi}{2}\right)}
\end{align*}
vanishes for $t\hat B= \pi(2n+1)$, $n \in \Z$. This fact is related to the nonexistence of an invertible holomorphic matrix solution as explained in \cite{EPT}. Indeed, denote by $T_0 = \lbrace \pi(2n+1)/\hat B \colon n \in \Z \rbrace$ the set of times for which $n^\md(k,t)$ vanishes at infinity.  We then conclude that
\begin{align*}
    m^{\text{mod}}(k,t_0) \propto k \, n^{\text{mod}}(k,t_0), \qquad t_0 \in T_0,
\end{align*}
as $k n^\md(k,t_0)$ is a symmetric solution of the model R-H problem and hence by Theorem \ref{modUniq} must be proportional to $m^\md(k,t_0)$. Note that for $t \not \in T_0$, $m^{\text{mod}}(k,t)$ and $n^{\text{mod}}(k,t)$ are linearly independent.

To show invertibility of $\id - \mathcal{S}^{\Sigma^{I}}_{u^{I}}$ we will need another linearly independent symmetric vector-valued function $\ell^\md(k)$, satisfying the jump condition from Theorem \ref{modUniq}, also for $t \in T_0$. However, instead of the normalization condition \eqref{eq:normcond}, we allow $\ell^\md(k) = O(k)$ as $k \to \infty$. As any holomorphic vector solution of the model problem that remains bounded at infinity must be already a multiple of $m^\md(k)$, this is a necessary requirement to obtain linear independence. In fact, the unboundedness of $\ell^\md(k)$ will be the reason why we need to consider the operator $\id - \mathcal{S}^{\Sigma^{I}}_{u^{I}}$ instead of $\id - \mathcal{C}^{\Sigma^{I}}_{u^{I}}$ (see condition \eqref{beta}). We choose for $\ell^\md(k)$:
\begin{align}\nonumber
     \ell_1^{\text{mod}}(k,t) &= \sqrt[4]{\frac{k^2 + a^2}{k^2+c^2}}\frac{\theta\left(\mathbf A(k) - \frac{\I\pi}{2}-\frac{\I t\hat B}{2}-\frac{\tau}{2}\right)\theta\left(\mathbf A(k)+\frac{\I \pi}{2} -\frac{\I t\hat B}{2}-\frac{\tau}{2}\right)}{{\theta\left(\mathbf A(k) - \frac{\I\pi}{2}-\frac{\tau}{2}\right)\theta\left(\mathbf A(k) + \frac{\I\pi}{2}-\frac{\tau}{2} \right)}},\\\label{ell}
\ell_2^{\text{mod}}(k,t) &= \sqrt[4]{\frac{k^2 + a^2}{k^2+c^2}}\frac{\theta\left(-\mathbf A(k) - \frac{\I\pi}{2}-\frac{\I t\hat B}{2}-\frac{\tau}{2}\right)\theta\left(-\mathbf A(k)+\frac{\I \pi}{2} -\frac{\I t\hat B}{2}-\frac{\tau}{2}\right)}
{{\theta\left(-\mathbf A(k) - \frac{\I\pi}{2}-\frac{\tau}{2}\right)\theta\left(-\mathbf A(k) + \frac{\I\pi}{2}-\frac{\tau}{2} \right)}}.
\end{align}
One can check that $\ell^\md(k)$ indeed satisfies the symmetry condition \eqref{eq:symcond} and grows linearly for $t\hat B \not = 2\pi n$, $n \in \mathbb{Z}$.
\begin{remark}
We do not claim uniqueness of $\ell^\emph{mod}(k)$ satisfying the above requirements. 
\end{remark}
\begin{remark}
A systematic approach to derive explicit formulas for $m^\emph{mod}(k)$, $n^\emph{mod}(k)$, and $\ell^\emph{mod}(k)$ can be found in \cite{PT} and is based on the reformulation of the model R-H problem from Theorem \ref{modUniq} as a scalar-valued R-H problem on the torus.
\end{remark}

Analogously to \eqref{solutionRHP1}, we define the vector-valued functions $n^{I}(k,t)$, $\ell^{I}(k,t)$ by inverting locally around $\pm\I c$, $\pm \I a$ by the appropriate local matrix-valued solutions $N(k)$, $Q(k)$. 
Consider now the following inhomogeneous R-H problem, with inhomogeneity $\lambda(k) \in L^2_s(\Sigma^{I};\Ct)$:

Find a vector-valued function $\mu(k)$, holomorphic for $k \in \C \setminus \Sigma^{I}$, such that
\begin{align}\nonumber
    \mu_+(k) &= \mu_-(k) v^I(k)+\lambda(k), \quad \mu_\pm(k) \in L^2_s(\Sigma^{I};\Ct)
    \\\label{mu}
    \mu(-k) &= \mu(k) \sigma_1, \quad k \in \C \setminus \Sigma^{I}
    \\\nonumber
    \lim_{k\to\infty} \mu(k) &= 0.
\end{align}
\begin{theorem}
The inhomogeneous R-H problem \eqref{mu} has a unique solution for any $\lambda(k) \in L_s^2(\Sigma^{I};\Ct)$.
\end{theorem}
\begin{proof}
We make the following ansatz:
\begin{align*}
    \mathcal C^{\Sigma^{I}}(\alpha)(k) n^{I}(k) + \mathcal C^{\Sigma^{I}}(\beta)(k) \ell^{I}(k) = \mu(k),
\end{align*}
where the two functions $\alpha(k)$, $\beta(k) \in L^1(\Sigma^{I})$ need to be determined. The jump condition in \eqref{mu} is then equivalent to
\begin{align}\label{Ansatz}
    \alpha(k)n^{I}_+(k) + \beta(k) \ell^{I}_+(k) = \lambda(k).
\end{align}
We make the following two key observations, related to the $t$-dependence:
\vskip 5pt
\begin{enumerate}
    \item $n^{I}_\pm(k,t)$ and $\ell^{I}_\pm(k,t)$ are uniformly bounded on $\Sigma^{I}$ for all $t \in \R$,
    \\
    \item $\det\big( n^{I}_\pm(k,t)^\top, \ell^{I}_\pm(k,t)^\top\big) = \dd(t) \not= 0$ for all $t \in \R$.
\end{enumerate}
\vspace{5pt}
Note that for point (1), it is necessary that we work with the bounded contour $\Sigma^{I}$, rather than with the unbounded contour $\Sigma^{II}$, and that we inverted locally by the matrices $N(k,t)$ and $Q(k,t)$ to avoid the fourth root singularities of the initial solutions  $n^\md_+(k,t)$, $\ell^\md_+(k,t)$. Point (2) follows from the fact that the determinant must be an entire function, which, due to the symmetries of $n^\md(k,t)$ and $\ell^\md(k,t)$, is in fact bounded as $k \to \infty$, hence  constant for $k \in \C$. Let us in the following again denote by the subscript $l$ (resp.~$r$) the left (resp.~right) limit to the contour $\Sigma^I$ at the origin. We have the four equalities
\begin{align*}
    n_{1,l}^I(0,t)&=n_{1,r}^I(0,t)\E^{\I t \hat B},
    \\
    n_{2,l}^I(0,t)&=-n_{1,r}^I(0,t),
    \\
    \ell_{1,l}^I(0,t)&=\ell_{1,r}^I(0,t)\E^{\I t \hat B},
    \\
    \ell_{2,l}^I(0,t)&=\ell_{1,r}^I(0,t),
\end{align*}
which imply
\begin{align*}
    \dd(t) = \det\big( n^{I}_l(0,t)^\top, \ell^{I}_l(0,t)^\top\big) = 2\E^{\I t\hat B} n_{1,r}^I(0,t) \ell_{1,r}^I(0,t).
\end{align*}
Using the fact that  $\theta(z) = 0$ if and only if $z \equiv \I\pi + \tau/2 \mod \Lambda$, and $\mathbf{A}_{r}(0) = \I\pi/2 - \tau/2$, it follows that $\dd(t) \not = 0$ for all $t \in\R$. In fact, $\dd(t)$ stays away from $0$ due to its periodicity. In the following we shall again suppress the $t$-dependence.

Together, these observations imply that given any $\lambda(k) \in L^2_s(\Sigma^{I};\Ct)$, there exist unique $\alpha(k)$, $\beta(k) \in L^2(\Sigma^{I})$ given by
\begin{align*}
\alpha(k) = -\frac{\det \big(\ell^{I}_+(k)^\top, \lambda(k)^\top \big)}{\dd}, \qquad
    \beta(k) = \frac{\det \big(n^{I}_+(k)^\top, \lambda(k)^\top \big)}{\dd},
\end{align*}
such that \eqref{Ansatz} holds. This would be even true for $\widetilde \lambda(k) \in L^2(\Sigma^{I};\Ct)$. However, not every pair $\alpha(k)$, $\beta(k)$ will lead to a solution of the inhomogeneous R-H problem \eqref{mu}. Indeed, because $\ell^{I}(k)$ grows linearly for $k \to \infty$, at least for $t\hat B \not = 2\pi n$, $n \in \mathbb{Z}$, we require $\mathcal C^{\Sigma^{I}}(\beta)(k)$ to decay quadratically for $k \to \infty.$ This translates to the simple condition
\begin{align}\label{beta}
    \int_{\Sigma^{I}} \beta(k) \, dk = 0.
\end{align}
It remains to check that provided $\lambda(k) \in L^2_s(\Sigma^{I};\Ct)$, the condition \eqref{beta} will be satisfied, so that the $\mu(k)$ defined in \eqref{Ansatz} will indeed be a solution of \eqref{mu}. First observe that
\begin{align}\nonumber
    \beta(-k) = \frac{\det \big(n^{I}_+(-k)^\top, \lambda(-k)^\top \big)}{\dd} &= \frac{\det \big(-\sigma_1n^{I}_+(k)^\top, \sigma_1 \lambda(k)^\top \big)}{\dd}  
    \\\label{symbeta}
    &= \frac{\det \big(n^{I}_+(k)^\top,  \lambda(k)^\top \big)}{\dd}= \beta(k).
\end{align}
Similarly, we have $\alpha(-k) = -\alpha(k)$. Performing the change of variables $s = -k$, \eqref{symbeta} implies
\begin{align*}
   \int_{\Sigma^{I}} \beta(k) \, dk =  -\int_{-\Sigma^{I}} \beta(-s) \, ds = 
   -\int_{\Sigma^{I}} \beta(s) \, ds \quad \Rightarrow \quad \int_{\Sigma^{I}} \beta(k) \, dk = 0
\end{align*}
where we used that the orientation of the contour $\Sigma^{I}$ is invariant under the transformation $k \to -k$. As $\alpha(k)$ is an odd and $\beta(k)$ is an even function, so are their Cauchy transforms, implying that the symmetry condition in \eqref{mu} holds. 

We see that we have constructed a solution $\mu(k)$ of \eqref{mu} for any $\lambda(k) \in L^2_s(\Sigma^{I};\Ct)$. To prove the uniqueness statement, consider the following homogeneous R-H problem:

Find a vector-valued function $\nu(k)$, holomorphic for $k \in \C \setminus \Sigma^{I}$, satisfying
\begin{align}\nonumber
    \nu_+(k) &= \nu_-(k) v^{I}(k), \qquad \nu_\pm(k) \in L^2_s(\Sigma^{I};\Ct),
    \\\label{vanishing}
    \nu(-k) &= \nu(k) \sigma_1, \quad k \in \C \setminus \Sigma^{I}
    \\\nonumber
    \lim_{k\to\infty} \nu(k) &= 0.
\end{align}
Note that up to local conjugations around $k =\pm \I a$, $\pm \I c$, \eqref{vanishing} is the vanishing problem for the model problem from Theorem \ref{modUniq}. We want to show that $\nu(k) \equiv 0$. To this end, define 
\begin{align}\label{defZeta}
    \zeta(k) = 
    \begin{cases}
\nu(k), & k \in \C \setminus \mathbb{D}_\cup
\\
\nu(k)A(k), & k \in \Da
\\
\nu(k)\sigma_1 A(-k)\sigma_1, & k \in \Dma
\\
\nu(k)Q(k), & k \in \Dc
\\
\nu(k)\sigma_1 Q(-k)\sigma_1, & k \in \Dmc
\end{cases}
\end{align}
Then $\zeta(k)$ satisfies the model R-H problem in the $L^2$-sense and $\zeta(k) \to 0$ as $k\to \infty$. Now observe that the jump matrix $v^\text{mod}(k)$ is locally constant on the three intervals $(-\I c, -\I a)$, $(-\I a, \I a)$ and $(\I a, \I c)$, and hence can be locally analytically continued. Thus, by locally deforming the jump contour, while fixing the points $\pm \I c$, $\pm \I a$, we see that $\zeta(k)$ must in fact have an analytic continuation across the three intervals $(-\I c, -\I a)$, $(-\I a, \I a)$ and $(\I a, \I c)$ from either side. Moreover, as $\nu(k)$ is locally holomorphic around $\pm \I c$ and $\pm \I a$, it follows from \eqref{defZeta} that $\zeta(k)$ has at most fourth root singularities as $k \to \kappa\in \lbrace -\I a, \I a, -\I c, \I c \rbrace$. In other words, $\zeta(k)$ is a vanishing solution of the model R-H problem from Theorem \ref{modUniq}, and thus must be constant equal to $0$, by the uniqueness of $m^{\text{mod}}(k)$. It follows that $\nu(k) \equiv 0$, finishing the proof.
\end{proof}
From Theorem \ref{ThmInh} it follows that $\id - \mathcal{S}^{\Sigma^{I}}_{u^{I}}$ is bijective and by the bounded inverse theorem, it follows that $(\id - \mathcal{S}^{\Sigma^{I}}_{u^{I}})^{-1}$ is a bounded operator from $L^2_s(\Sigma^{I};\Ct)$ to itself. By Corollary \ref{LongCor} the same conclusion can be drawn for $\id - \mathcal{S}^{\Sigma^{II}}_{u^{I}}$ and its inverse.
\subsection{Comparison of resolvents}
We can now use the results from the previous sections, to prove that $m^{(2)}(k)$ and $m^\md(k)$ are asymptotically close to one another as $t \rightarrow \infty$.

 In fact, as $u^{I}(k)$ is periodic in time, we can conclude that the continuous family of operators $\id - \mathcal{S}^{\Sigma^{II}}_{u^I}$ is uniformly invertible
\beq\nn
\Vert (\id - \mathcal{S}^{\Sigma^{II}}_{u^I})^{-1} \Vert_{L^2_s(\Sigma^{II};\Ct)} \leq C. 
\eeq
Here, $C$ can be chosen locally uniformly in the parameter $\xi = \frac{x}{12t}$ (see remark below and Appendix A).
\begin{remark}\label{rmkxi}
It should be emphasized that $u^{I}(k)$, $u^{II}(k)$  do not only depend on $k$ and time, but also dependent on $\xi = \frac{x}{12t}$. As shown in Appendix A, we can choose the contour $\Sigma^{II}$ such that it does not depend on the parameter $\xi$, as long as $\xi$ stays in some compact subinterval of $(-\frac{c^2}{2}, \frac{c^2}{3})$. As the jump matrices are continuous functions of $\xi$, we can vary $\xi$ when letting $t \rightarrow \infty$ and all estimates would still hold. In our subsequent computations we suppress the $\xi$-dependence as it does not change the asymptotics as long $\xi$ stays in $(-c^2/2+\varepsilon, c^2/3- \varepsilon)$, for some $\varepsilon > 0$.
\end{remark} 
From now on we abbreviate the norms $\Vert . \Vert_{L^p_s(\Sigma^{II};\Ct)}$ and $\Vert . \Vert_{L^p(\Sigma^{II};\Ctt)}$ by $\Vert . \Vert_p$, for $p \in [1, \infty]$. We note that 
\beq\nn
\begin{gathered}
\Vert v^{I} - v^{II} \Vert_\infty = \Vert u^{I} - u^{II} \Vert_\infty = O(t^{-1}) 
\\
\Vert v^{I} - v^{II} \Vert_2 = \Vert u^{I} - u^{II} \Vert_2 = O(t^{-1})
\end{gathered}
\eeq
and
\beq\nn
\Vert \mathcal{S}^{\Sigma^{II}}_{u^{I}} - \mathcal{S}^{\Sigma^{II}}_{u^{II}} \Vert_2 = \Vert \mathcal{S}^{\Sigma^{II}}_{u^{I}-u^{II}} \Vert_2 = O(t^{-1}).
\eeq
As the set of bounded invertible operators is open in the operator norm topology, we conclude that for $t$ large enough the operator $\id-\mathcal{S}^{\Sigma^{II}}_{u^{II}}$ must also be uniformly invertible. Denote by $\phi^{II}(k)$ the unique solution of
\beq \label{SIE2}
(\id - \mathcal{S}^{\Sigma^{II}}_{u^{II}})\phi^{II}(k) = \mathcal{S}^{\Sigma^{II}}_-((1 \hspace{7pt} 1)u^{II}).
\eeq
The following computations shows that $\phi^{I}(k)$ and $\phi^{II}(k)$ are in fact also close to one another in $L^2_s(\Sigma^{II};\Ct)$:
\beq\nn
\Vert \phi^{I} - \phi^{II} \Vert_2
= \Vert (\id - \mathcal{S}^{\Sigma^{II}}_{u^I})^{-1} \mathcal{S}^{\Sigma^{II}}_-((1 \hspace{7pt} 1) u^I) - (\id - \mathcal{S}^{\Sigma^{II}}_{u^{II}})^{-1} \mathcal{S}^{\Sigma^{II}}_-((1 \hspace{7pt} 1) u^{II}) \Vert_2
\eeq
\beq
\nonumber
\Vert 
[(\id - \mathcal{S}^{\Sigma^{II}}_{u^{I}})^{-1}- (\id - \mathcal{S}^{\Sigma^{II}}_{u^{II}})^{-1}] \mathcal{S}^{\Sigma^{II}}_-((1 \hspace{7pt} 1)u^{I}) \Vert_2+
\Vert (\id - \mathcal{S}^{\Sigma^{II}}_{u^{II}})^{-1}\mathcal{S}^{\Sigma^{II}}_-((1 \hspace{7pt} 1) (u^{I}-u^{II})) \Vert_2 
\eeq	
\beq
\nonumber
\leq
C_1 \Vert (\id - \mathcal{S}^{\Sigma^{II}}_{u^{I}})^{-1}\mathcal{S}^{\Sigma^{II}}_{u^{I}-u^{II}}(\id - \mathcal{S}^{\Sigma^{II}}_{u^{II}})^{-1} \Vert_2+
C_2 \Vert u^{I} - u^{II} \Vert_2  = O(t^{-1})
\eeq
where we use the second resolvent formula. Furthermore, $\Vert \phi^{I,II} \Vert_2$ are uniformly bounded by the uniform invertibility of the corresponding singular integral operators, as well as the uniform boundedness of $\Vert u^{I,II} \Vert_2$.
Now, developing $1/(s-k)$ into a truncated Neumann series
\beq\nn
\dfrac{1}{s-k} = -\dfrac{1}{k} -\dfrac{s}{k^2}-\dfrac{s^2}{k^3}\dfrac{1}{1-s/k}
\eeq
and taking into account the exponential decay of the matrices $u^{I}(k), u^{II}(k)$ at infinity, one obtains for $k \rightarrow \infty$, such that $|1-s/k| \geq \varepsilon > 0$ for all $s \in \Sigma^{II}$,
\beq\nn
\begin{split}
m^\md(k) = (1 \hspace{7pt} 1) - \dfrac{1}{2\pi \I k} \int_{\Sigma^{II}} (\phi^{I}(s)+(1 \hspace{7pt} 1)) u^{I}(s) \ ds 
\\
- \dfrac{1}{2\pi \I k^2}
\int_{\Sigma^{II}} (\phi^{I}(s)+(1 \hspace{7pt} 1)) u^{I}(s)  s \ ds + O(k^{-3})
\end{split}
\eeq
and
\begin{align*}
\begin{split} 
m^{(2)}(k) = (1 \hspace{7pt} 1) - \dfrac{1}{2\pi \I k} \int_{\Sigma^{II}} (\phi^{II}(s)+(1 \hspace{7pt} 1)) u^{II}(s) \ ds 
\\
- \dfrac{1}{2\pi \I k^2}
\int_{\Sigma^{II}} (\phi^{II}(s)+(1 \hspace{7pt} 1)) u^{II}(s)  s \ ds + O(k^{-3}).
\end{split}
\end{align*}
Next we compute
\beq\nn
\begin{gathered}
\Big| \int_{\Sigma^{II}} (\phi^I(s)+(1 \hspace{7pt} 1)) u^I(s) \ ds - \int_{\Sigma^{II}} (\phi^{II}(s)+(1 \hspace{7pt} 1)) u^{II}(s) \ ds \Big|
\\
\leq \Big| \int_{\Sigma^{II}} (\phi^I(s)-\phi^{II}(s)) u^I(s) \ ds \Big| + \Big| \int_{\Sigma^{II}} \phi^{II}(s) (u^I(s)-u^{II}(s)) \ ds \Big|
\\
+ \ \Big| \int_{\Sigma^{II}} (1 \hspace{7pt} 1)(u^I(s) - u^{II}(s)) \ ds \Big| = O(t^{-1})
\end{gathered}
\eeq
and analogously
\beq\nn
\Big| \int_{\Sigma^{II}} (\phi^I(s)+(1 \hspace{7pt} 1)) u^I(s) s\ ds - \int_{\Sigma^{II}} (\phi^{II}(s)+(1 \hspace{7pt} 1)) u^{II}(s) s\ ds \Big| = O(t^{-1}).
\eeq
Hence, we can conclude that 
\beq\label{comparison}
m^{(2)}(k) = m^\md(k) + \dfrac{Er_1(t,\xi)}{k} + \dfrac{Er_2(t,\xi)}{k^2}+O(k^{-3})
\eeq
with
\begin{align}\nn
Er_1(t,\xi) &= \frac{1}{2\pi \I} \int_{\Sigma^{II}} (\phi^{I}(s)+(1 \hspace{7pt} 1)) u^{I}(s)-(\phi^{II}(s)+(1 \hspace{7pt} 1)) u^{II}(s) \ ds = O(t^{-1}),
\\\label{error terms}
Er_2(t,\xi) &= \frac{1}{2\pi \I} \int_{\Sigma^{II}} \big[(\phi^{I}(s)+(1 \hspace{7pt} 1)) u^{I}(s)-(\phi^{II}(s)+(1 \hspace{7pt} 1)) u^{II}(s)\big] s \ ds = O(t^{-1}),
\end{align}
where the $O(t^{-1})$ estimates for the error terms hold uniformly as $t \to \infty$ for $\xi \in (-\tfrac{c^2}{2}+\varepsilon, \tfrac{c^2}{3}-\varepsilon)$, $\varepsilon > 0$. Moreover, as $m^{(2)}(k)$ and $m^\md(k)$ satisfy condition \eqref{eq:symcond}, we obtain from \eqref{comparison}:
\beq\nn
\frac{Er_1(t,\xi)}{-k} = \frac{Er_1(t,\xi)\sigma_1}{k}, \qquad \frac{Er_2(t,\xi)}{(-k)^2} = \frac{Er_2(t,\xi)\sigma_1}{k^2},
\eeq
implying that $Er_1(t,\xi)$ and $Er_2(t,\xi)$ can be written as
\beq\nn
Er_1(t,\xi) = er_1(t,\xi)(1 \hspace{3pt} -1), \qquad Er_2(t,\xi) = er_2(t,\xi)(1 \hspace{7pt} 1),
\eeq
with scalar-valued $er_1(t,\xi)$, $er_2(t,\xi) = O(t^{-1})$, uniformly for $\xi \in (-\tfrac{c^2}{2}+\varepsilon, \tfrac{c^2}{3}-\varepsilon)$, $\varepsilon > 0$.

We will now make use of the formula \eqref{mproduct}
\beq \nonumber
m_1(k,x,t) m_2(k,x,t) = 1 + \dfrac{q(x,t)}{2k^2} + O(k^{-4}).
\eeq
There is no need to trace back our deformation and conjugation steps, as we are only interested in the asymptotics of $m(k)$ at infinity, and conjugation by $\E^{(t\Phi(k)/2 -\I t g(k))\sigma_3}F(k)^{\sigma_3}$ will not change the product $m^{(2)}_1(k) m^{(2)}_2(k)$. Hence we can conclude that
\begin{align}\nonumber
m_1(k) m_2(k) &= m^{(2)}_1(k) m^{(2)}_2(k) 
\\\label{final estimate}
&= m^\md_1(k) m^\md_2(k) + \frac{er_1(t,\xi)}{k}[m_2^{\text{mod}}(k)-m_1^{\text{mod}}(k)] 
\\\nonumber
&\quad - \frac{er_1(t,\xi)^2}{k^2} + \frac{er_2(t,\xi)}{k^2}[m_2^{\text{mod}}(k)+m_1^{\text{mod}}(k)]+O(k^{-4}).
\end{align}
Note that the remaining error term is of order $O(k^{-4})$, as both sides of the equation are even functions in $k$ due to \eqref{eq:symcond}. As $m^{(2)}(k) \to (1 \hspace{7pt} 1)$ for $k \to \infty$, we have that
\beq\nn
m_2^{\text{mod}}(k)-m_1^{\text{mod}}(k) = O(k^{-1}).
\eeq
We thus conclude from \eqref{error terms}, \eqref{final estimate} that
\begin{align} \label{mEstimate}
    m_1(k) m_2(k) = m^\md_1(k) m^\md_2(k) + \frac{1}{k^2}O(t^{-1}).
\end{align}
For the solution $q(x,t)$ of the KdV equation we obtain from \eqref{mEstimate}
\beq\nn
q(x,t) = q^{\text{\text{mod}}}(x,t) + O(t^{-1})
\eeq
where 
\beq\nn
m_1^{\text{\text{mod}}}(k)m_2^{\text{\text{mod}}}(k) = 1 + \dfrac{q^{\text{\text{mod}}}(x,t)}{2k^2} + O(k^{-4}).
\eeq
In \cite{EGT} it has been shown that $q^{\text{\text{mod}}}(x,t)$ has the form of a periodic Its--Matveev solution modulated by the parameter $\xi$. A general theorem summarizing the above argumentation in the abstract setting is given in Appendix B. 
\section{discussion}
The main difference in our nonlinear steepest descent analysis compared to the usual one (see \cite{PD}, \cite{DZ}, \cite{DZNLS}), has been the avoidance of a small norm R-H problem. Instead, to obtain invertibility of the associated singular integral operators, we relied on the results from \cite[Sect.~2]{DZsob}. We also include a general theorem in Appendix B on obtaining the resolvent estimates found in Section 5.2 assuming the invertibility of the corresponding singular integral operators. Note however, that in our application we still needed invertible local versions of the model matrix solution around the points $\pm \I a$, $\pm \I c$.

An issue not covered here in detail, is the computation of a full asymptotic expansion of the R-H solution, as it is done in \cite{DKMVZ} for the case of orthogonal polynomials.  This works analogously in our case, as we have a shifted Neumann series given by:
\beq\nn
(\id - \mathcal{S}^{\Sigma^{II}}_{u^{II}})^{-1} = \sum_{n=0}^\infty \Big[ (\id - \mathcal{S}^{\Sigma^{II}}_{u^I})^{-1} ( \mathcal{S}^{\Sigma^{II}}_{u^{II}} - \mathcal{S}^{\Sigma^{II}}_{u^I}) \Big]^n (\id - \mathcal{S}^{\Sigma^{II}}_{u^I})^{-1}.
\eeq
This in itself is not enough to write down an expansion of the solution to the singular integral equation \eqref{SIE2} in powers of $t^{-1}$. We also need an expansion of 
\beq\nn
\mathcal{S}^{\Sigma^{II}}_{u^{II}-u^I} = \mathcal{S}^{\Sigma^{II}}_{u^{II}} - \mathcal{S}^{\Sigma^{II}}_{u^{I}}
\eeq
which is equivalent to an expansion of
\beq\nn
u^{II}(k) - u^{I}(k) = A^{-1}(k) - N^{-1}(k)
\eeq
on $\partial \Da$, as $u^{II}(k)-u^{I}(k) = o(t^{-r})$ for $r \in \N$ on the rest of the contour. To this end we make use of the full expansion of the Airy functions \cite[Ch.~9]{OLVER} in powers of $w^{-3/2}$, which translates to an expansion in $t^{-1}$. This results in 
\beq\nn
\phi^{II}(k,x,t) = \sum_{j=0}^r \dfrac{\phi^{II}_j(k,x,t)}{t^j} + Er(k,x,t)
\eeq
for $r \in \N$ and $\xi$ fixed, with $\phi^{II}_j(k,x,t)$ being a periodic function in $t$, $\phi^{II}_0 = \phi^{I}$ and $\Vert Er(\, \cdot \,,x,t) \Vert_2 = O(t^{-r-1})$. Consequently, one obtains a similar expansion of $m^{(2)}(k,x,t)$ in terms of $t^{-1}$ and $k^{-1}$,
\beq\nn
m^{(2)}(k,x,t) = m^\md(k,x,t) + \sum_{i=1}^r \sum_{j=1}^n \dfrac{c_{ij}(t,\xi)}{k^i t^j} +O(k^{-r-1}t^{-1} ) + O(k^{-1}t^{-n-1}).
\eeq
where $c_{ij}(t,\xi)$ are periodic in $t$ for fixed $\xi$. Hence, we see that the existence of an asymptotic expansion of the R-H solution follows  from an asymptotic expansion of the jump matrices, just as with the traditional small norm approach. It is interesting to observe, that while the solution to the model problem gives us the leading asymptotics, all other expansion terms are derived from the local parametrices around $\pm \I a$, i.e.~the upper edge of the first spectral band.

Another future challenge would be characterizing those R-H problems that do not admit an invertible holomorphic matrix-valued solution. Note that precisely for $t \hat B(\xi)=n\pi$, $n \in \Z$, the model R-H problem has an additional symmetry:
\beq\nn
v^\md(k) = \sigma_1 v^\md(k) \sigma_1
\eeq  
where $\sigma_1$ is the first Pauli matrix. In particular, one can check that for those values of $t$, $m^{\text{mod}}(k)\sigma_1$ is also a symmetric solution of the model problem. From  uniqueness it follows that $m^\text{mod}(k)\sigma_1 = m^\text{mod}(k)$, or equivalently $m^\text{mod}_1(k) = m^\text{mod}_2(k)$. This makes it easier to satisfy the equation $m^\text{mod}(k) = (0 \hspace{7pt} 0)$ for some $k$, which is related to the nonexistence of an invertible matrix-valued model solution (see \cite[Sect.~3]{EPT}). Indeed for odd $n$, $m^\text{mod}_\pm(0) = (0 \hspace{7pt} 0)$ holds and no holomorphic invertible matrix-valued solution exists. It would be interesting to explore the question whether such symmetric problems have a distinguished role in the R-H analysis of integrable equations. 

\appendix
\numberwithin{equation}{section}
\section{Uniformity of operator bounds}

Observe that the contour $\Sigma$ of the model R-H problem with the exponentially converging matrices depends on the parameter $\xi = \frac{x}{12t}$ via the point $a = a(\xi) \in (0, c)$. However, it is possible to make the contour $\Sigma^{II} = (\Sigma \setminus \mathbb{D}_\cup) \cup \partial \mathbb{D}_\cup$ at least locally in $\xi$, independent of $\xi$. To see this, note that while we for simplicity always assumed $\I a$ to be the center of $\Da$, this is not essential. Furthermore, we can always choose the rays emanating from $\I a$ to hit the boundary of the disc at the same points. This then allows us to choose the rest of the contour $\Sigma^{II}$ independent of $\xi$ as long as $\I a$ stays in the interior of $\Da$, which is now chosen to be, at least locally in $\xi$, independent of $\I a$.

\begin{figure}[H]
\begin{center}
\begin{tikzpicture}

\draw (0,0) circle (2.5 cm);

\draw (0,0) -- (1.75, 1.75);

\draw (0,0) -- (-1.75, 1.75);

\draw (0, -2.5) -- (0, 2.5);

\draw [dashed] (0, -2) -- (1.75, 1.75);

\draw [dashed] (0, -2) -- (-1.75, 1.75);
 
\node at (0.5,-0.1) {$\I a(\xi_0)$};

\node at (0.5,-2.1) {$\I a(\xi)$};

\node at (0,0) [circle,fill,inner sep=1pt]{};

\node at (0,-2) [circle,fill,inner sep=1pt]{};

\end{tikzpicture}
\end{center}
\caption{The disc $\Da = \mathbb D_{\I a(\xi_0)}$, with the inner rays depending on $\xi$.}
\end{figure}
This greatly simplifies the analysis, as we now have to deal with only one $\xi$-independent Hilbert space $L^2_s(\Sigma^{II};\Ct)$. While we might not be able to choose $\Da$ as large as possible because of the constraint that $k \rightarrow w$ should be bijective, we certainly can cover any compact interval contained in $(0, \I c)$ with finitely many discs. Hence all our estimates will be uniform, as long as $\xi$ stays in some compact subinterval of $(-c^2/2, c^2/3)$. 

Another issue neglected in the main text is the uniform boundedness of
\beq\nn
(\id - \mathcal{S}^{\Sigma^{II}}_{u^{I, II}})^{-1}.
\eeq
Note that for two operators $O$ and $P$, where $O$ is invertible and $P$ is some perturbation of $O$, we have the estimates
\beq \label{invB}
\Vert (O+P)^{-1} \Vert \leq \Vert O^{-1} \Vert \dfrac{1}{1- \Vert O^{-1} \Vert \Vert P \Vert}
\eeq
and similarly
\beq\nn
\Vert (O+P)^{-1} \Vert \geq \Vert O^{-1} \Vert \dfrac{1}{1+ \Vert O^{-1} \Vert \Vert P \Vert},
\eeq
whenever $\Vert O^{-1} \Vert \Vert P \Vert < 1$.
This implies continuity of the norm of the inverse. In particular, we can conclude that for $(\xi, t) \in K \times [T_1, T_2]$, where $K \subset (-c^2/2, c^2/3)$ is compact and $T_1 < T_2$, we have the estimate
\beq\nn
\Vert (\id - \mathcal{S}^{\Sigma^{II}}_{u^{I}})^{-1} \Vert_2 \leq C < \infty.
\eeq
By periodicity of $u^{I}(k)$ in time this estimate can be extended to $t \in \mathbb{R}$. Analogously, because of $\Vert u^{I} - u^{II} \Vert_\infty = O(t^{-1})$ we get 
\beq\nn
\Vert (\id - \mathcal{S}^{\Sigma^{II}}_{u^{II}})^{-1} \Vert_2 \leq C' < \infty.
\eeq
for $t$ large enough.

\section{A General theorem}
We now mention a theorem generalizing the argumentation given in the proof of the main result (c.f \cite[Cor.~7.108]{PD}, \cite[Ch.~3]{FIKN}, \cite[Prop.~4.4]{XZ}). Let $\Gamma$ be an oriented contour, such that the associated Cauchy operators $\mathcal{C}_\pm^\Gamma$ are bounded operators from $L^2(\Gamma)$ to itself. Necessary and sufficient conditions on the contour for the above statement to hold can be found in \cite{BK97} or \cite{JL0}.  Furthermore, let an $n \times n$ matrix-valued function $v \in \id + L^2(\Gamma;\Cnn)$ be given, such that $v^{-1} \in  \id + L^2(\Gamma;\Cnn)$. We associate to $v$ a factorization data $u=(u^+, u^-) \in L^2(\Gamma;\Cnn) \cap L^\infty(\Gamma;\Cnn)$, such that $v = (\id - u^-)^{-1} (\id + u^+)$ on the contour $\Gamma$. Note that the factorization data is nonunique, but always exists, as one can choose $u^{-} = 0$ and $u^{+} = v-\id$, as it is done in the main text. For any factorization data we define a singular integral operator
\beq \nonumber
\mathcal{C}^\Gamma_{u}: L^2(\Gamma;\Cn) \rightarrow L^2(\Gamma;\Cn), \hspace{7pt} \phi \mapsto \mathcal{C}^\Gamma_+(\phi u^-) + \mathcal{C}^\Gamma_-(\phi u^+).
\eeq
Again, we are interested in solutions of the R-H problem on the contour $\Gamma$ with jump matrix $v$. The normalization for the vector-valued solution $m$ is assumed to take the simple form
\beq\nn
\lim\limits_{k \rightarrow \infty} m(k)= m_\infty \in \C^n
\eeq
where the limit is taken such that $|1-s/k| \geq \varepsilon > 0$ for all $s \in \Gamma$ and some positive constant $\varepsilon$. As before the above R-H problem is equivalent to the following singular integral equation
\beq \label{singInt}
\big(\id - \mathcal{C}^\Gamma_{u} \big) \phi = \mathcal{C}^\Gamma_u(m_\infty)
\eeq
where $m$ can be obtained by the formula
\begin{align}\label{defOfm}
\begin{split}
m(k)&= m_\infty + \dfrac{1}{2\pi \I} \int_\Gamma (\phi(s)+m_\infty)(u^+(s) + u^-(s))\dfrac{ds}{s-k}
\\
&= m_\infty + \mathcal{C}^\Gamma((\phi + m_\infty)(u^+ + u^-))(k).
\end{split}
\end{align}
Indeed, assume $\phi$ satisfies \eqref{singInt} and define $m$ as above. Then
\begin{align*}
\begin{split}
m_+ &= m_\infty + \mathcal{C}^\Gamma_+((\phi + m_\infty)(u^+ + u^-))
\\
&= m_\infty + \mathcal{C}^\Gamma_-(\phi u^+) +\phi u^+ + \mathcal{C}^\Gamma_+(\phi u^-) + \mathcal{C}^\Gamma_-(m_\infty u^+) + m_\infty u^+ + \mathcal{C}^\Gamma_+(m_\infty u^-)
\\
&= m_\infty(\id + u^+) + \underbrace{\mathcal{C}^\Gamma_u(\phi) + \mathcal{C}^\Gamma_u(m_\infty)}_\phi + \phi u^+
\\
&= (m_\infty+\phi)(\id + u^+)
\end{split}
\end{align*}
where we used $\mathcal{C}^\Gamma_+ - \mathcal{C}^\Gamma_- = \id$. Analogously one computes
\beq\nn
m_- = (m_\infty + \phi)(\id- u^-),
\eeq
which then implies
\beq\nn
m_+ = m_-(\id-u^-)^{-1} (\id + u^+) = m_- v.
\eeq
Hence, $m$ is a solution of the R-H problem with $\lim_{k \rightarrow \infty} m(k) = m_\infty$. Conversely, assume $m$ is a solution to the R-H problem. Then by the Sokhotski--Plemelj formula for additive R-H problems, $m$ can be written as
\beq \nonumber
m = m_\infty + \mathcal{C}^\Gamma(m_+(\id - v^{-1}))=m_\infty + \mathcal{C}^\Gamma(m_-(v-\id)). 
\eeq
Define 
\beq\nn
\phi = m_-(\id -u^-)^{-1} - m_\infty = m_+(\id +u^+)^{-1} - m_\infty.
\eeq
Then relation \eqref{defOfm} is fulfilled.
From the definition of $\phi$ it follows that
\begin{align*}
\begin{split}
m_+ &= (\phi + m_\infty)(\id +u^+),
\\
m_- &= (\phi + m_\infty)(\id -u^-).
\end{split}
\end{align*}
Meanwhile, \eqref{defOfm} together with $\mathcal{C}^\Gamma_+ - \mathcal{C}^\Gamma_- = \id$, imply as before 
\begin{align*} 
\begin{split}
m_+ &= m_\infty(\id + u^+) + \mathcal{C}^\Gamma_u(\phi) + \mathcal{C}^\Gamma_u(m_\infty)+\phi u^+,
\\
m_- &= m_\infty(\id - u^-) + \mathcal{C}^\Gamma_u(\phi) + \mathcal{C}^\Gamma_u(m_\infty) -\phi u^-.
\end{split}
\end{align*}
Comparing the two expressions of either $m_+$ or $m_-$ results in \eqref{singInt}. We are now in a position to state a theorem generalising the arguments given in Section 5.
\begin{theorem}
Let $\Gamma$ be a contour such that the Cauchy operators $\mathcal{C}^\Gamma_\pm$ are bounded operators from $L^2(\Gamma)$ to itself with operator bounds less than $C>0$, and let for $i = 1,2$
\beq\nn
v_i : \R_+ \rightarrow \id + L^2(\Gamma;\Cnn), \hspace{7pt} t \mapsto v_i(t) = v_i(t,k)
\eeq
together with a factorization
\beq\nn
v_i(t) = (\id - u_i^-(t))^{-1}(\id + u_i^+(t)), \hspace{7pt} t > 0
\eeq
be given, such that $v_i^{-1}(t) \in \id + L^2(\Gamma;\Cnn)$ and $u_i^{\pm}(t) \in L^2(\Gamma;\Cnn) \cap L^\infty(\Gamma;\Cnn)$. Furthermore, assume that the operator $\id - \mathcal{C}^\Gamma_{u_1}$ is invertible for all $t > 0$ with
\beq\nn
\Vert (\id - \mathcal{C}^\Gamma_{u_1})^{-1} \Vert_2 \leq \rho(t)
\eeq
and
\beq\nn
\Vert u_1^\pm - u_2^\pm \Vert_{2} \leq \epsilon(t), \hspace{10pt} \Vert u_1^\pm - u_2^\pm \Vert_{\infty} \leq \delta(t),
\eeq 
where $\rho(t), \epsilon(t)$ and $\delta(t)$ are given positive functions and $\Vert . \Vert_p := \Vert  . \Vert_{L^p(\Gamma)}$, $p \in [1,\infty]$, where the norm is naturally generalized to matrix and vector functions. Then $\id - \mathcal{C}^\Gamma_{u_2}$ is also invertible as long as $C \rho(t) \delta(t) < 1$ with
\beq\nn
\Vert (\id - \mathcal{C}^\Gamma_{u_2})^{-1} \Vert_2 \leq \dfrac{\rho(t)}{1-C\rho(t)\delta(t)}.
\eeq 
For $t >0$ such that $C \rho(t) \delta(t) < 1$, denote by $\phi_{1,2}$ the unique solutions of 
\beq\nn
(\id - \mathcal{C}^\Gamma_{u_{1,2}}) \phi_{1,2} = \mathcal{C}^\Gamma_{u_{1,2}}(m_\infty)
\eeq
for some fixed $m_\infty \in \C^n$. Then
\beq\nn
\Vert \phi_1 - \phi_2 \Vert_{2} \leq \dfrac{2C\rho(t) \epsilon(t)}{1-C\rho(t)\delta(t)} \Vert m_\infty \Vert_{\infty} 
+ \dfrac{2C^2 \rho^2(t) \delta(t)}{1-C\rho(t)\delta(t)} \Vert m_\infty \Vert_{\infty} \Vert u_1 \Vert_{2}.
\eeq
Now, assume that the $i$-th moments of $u^\pm_1$ and $u^\pm_2$ exist, in the sense that
\beq\nn
\Vert u^{\pm}_j(k) k^i \Vert_{p} \leq \infty
\eeq
for $j = 1,2$, $p=1,2$ and $i = 0, \dots, r$. Then for the vector solutions $m_1$ and $m_2$ associated to $\phi_1$ and $\phi_2$ via \eqref{defOfm}, we have the expansion
\beq\nn
m_j(k) = m_\infty - \sum_{i = 1}^{r} \dfrac{1}{k^i} \int_\Gamma (\phi_j(s)+m_\infty) (u_j^+(s)+u_j^-(s)) s^{i-1}ds + O(k^{-r-1})
\eeq
for $j=1,2$ where $k \rightarrow \infty$ such that $|1-s/k| \geq c>0$ for $s \in \Gamma$.
Furthermore, if
\beq\nn
\Vert u^{\pm}_1(k) k^i - u^{\pm}_2(k) k^i  \Vert_{2} \leq \rho_i(t) 
\eeq
\beq\nn
\Vert u^{\pm}_1(k) k^i - u^{\pm}_2(k) k^i  \Vert_{1} \leq \sigma_i(t)
\eeq
for $i = 0, \dots, r-1$, then
\beq\nn
m_1(k) - m_2(k) = \sum_{i = 1}^{r} \dfrac{c_i}{k^i} + O(k^{-r-1})
\eeq
with
\begin{align*}
|c_i| \leq   \Vert \phi_1 - \phi_2 \Vert_{2} \Vert (u_1^+(k) + u_1^-(k))k^{i-1} \Vert_{2} +  2 \Vert \phi_2 \Vert_{2} \rho_{i-1}(t) + 2\Vert m_\infty \Vert_\infty \sigma_{i-1}(t).
\end{align*}
Moreover,
\beq
\nonumber
|m_1(k) - m_2(k)| \leq \textup{dist}(k,\Gamma)^{-1}\Big[\Vert \phi_1 - \phi_2 \Vert_{2} \Vert (u_1^+(k) + u_1^-(k))\Vert_{2} +  2 \Vert \phi_2 \Vert_{2} \rho_0(t)
\eeq
\beq\nn
+ 2\Vert m_\infty \Vert_\infty \sigma_0(t)\Big].
\eeq
\end{theorem}

\begin{proof}
The statement concerning the existence and bound of $(\id - \mathcal{C}_{u_2}^\Gamma)^{-1}$ follows directly from formula \eqref{invB} in Appendix A. The estimates for $\Vert \phi_1 - \phi_2 \Vert_{2}$ and $c_i$ can be computed as it is done in the proof of our main result concerning the KdV equation, where we identify $u_1$ with $u^I$ and $u_2$ with $u^{II}$. The last estimate is obtained similarly by bounding $(k-s)^{-1}$ by $\mbox{dist}(k,\Gamma)^{-1}$ instead of writing down the Neumann series. 
\\
\end{proof}
\begin{remark}
With the identification of $u_1$ with $u^{II}$ and $u_2$ with $u^{I}$ one obtains analogously the estimates
\beq\nn
\Vert \phi_1 - \phi_2 \Vert_{2} \leq 2C\rho(t) \epsilon(t) \Vert m_\infty \Vert_{\infty} 
+ \dfrac{2C^2 \rho^2(t) \delta(t)}{1-C\rho(t)\delta(t)} \Vert m_\infty \Vert_{\infty} \Vert u_2 \Vert_{2},
\eeq
\vspace{2pt}
\beq\nn
|c_i| \leq   \Vert \phi_1 - \phi_2 \Vert_{2} \Vert (u_2^+(k) + u_2^-(k))k^{i-1} \Vert_{2} +  2 \Vert \phi_1 \Vert_{2} \rho_{i-1}(t)
+ 2\Vert m_\infty \Vert_\infty \sigma_{i-1}(t),
\eeq
\vspace{2pt}
\beq \nonumber
|m_1(k) - m_2(k)| \leq \textup{dist}(k,\Gamma)^{-1}\Big[\Vert \phi_1 - \phi_2 \Vert_{2} \Vert (u_2^+(k) + u_2^-(k))\Vert_{2} +  2 \Vert \phi_1 \Vert_{2} \rho_0(t)
\eeq
\beq\nn
+ 2\Vert m_\infty \Vert_\infty \sigma_0(t)\Big].
\eeq
\end{remark}

\noindent{\bf Acknowledgments.} The author thanks Iryna Egorova and Gerald Teschl for fruitful discussions and helpful remarks which are scattered throughout this work.

\end{document}